\documentclass[11pt]{article}
\usepackage[utf8]{inputenc}
\usepackage{lmodern}
\usepackage{subfiles}
\usepackage{enumitem}
	\setenumerate{label={\normalfont(\roman*)}, itemsep=0em} % or \upshape
        
\usepackage{amsfonts}
\usepackage{amsthm}
\usepackage{amsmath}
\usepackage{amssymb}
\usepackage{amscd}
\usepackage{thmtools, thm-restate}
\usepackage{mathrsfs}
\usepackage{mathtools}
\usepackage{bm}
\usepackage{esint}

\usepackage[margin=3cm]{geometry}
\usepackage{indentfirst}
\usepackage{graphicx}
\usepackage{graphics}
\usepackage{lscape}
\usepackage{tikz-cd}
\usepackage{tikz,pgf}
	\usetikzlibrary{arrows}
\usepackage{color}
\usepackage{pict2e}
\usepackage{epic}
\usepackage{epstopdf}
\usepackage{titlesec, titlefoot}
	\titleformat{\section}[block]{\Large\bfseries\filcenter}{\thesection}{1em}{}
\usepackage{commath}
\usepackage{float}
\usepackage{caption}
\usepackage{etoolbox}
\usepackage[affil-it]{authblk}
\usepackage{combelow}
\usepackage{bbm}

\usepackage[hidelinks]{hyperref}
\hypersetup{bookmarksopen=true} 
\usepackage{hypcap}

\graphicspath{{./Pictures/}}
\allowdisplaybreaks

\usepackage{listings}
\usepackage{fancybox}

\theoremstyle{plain}

\renewcommand*\thesection{\arabic{section}}
\numberwithin{equation}{section} 

\theoremstyle{plain}
\newtheorem{thm}{Theorem}

\newtheorem{lemma}[thm]{Lemma}

\newtheorem{theorem}{Theorem}

\theoremstyle{definition}

\newtheorem{remark}[thm]{Remark}

\newtheorem{definition}[thm]{Definition}
\newtheorem{claim}[thm]{Claim}

\newcommand{\thistheoremname}{}
\newtheorem{genericthm}[equation]{\thistheoremname}

\newcommand{\thistheoremnames}{}
\newtheorem*{genericthms}{\thistheoremnames}
\newenvironment{para*}[1]
  {\renewcommand{\thistheoremnames}{#1}%
   \begin{genericthms}}
  {\end{genericthms}}

\expandafter\let\expandafter\oldproof\csname\string\proof\endcsname
\let\oldendproof\endproof
\renewenvironment{proof}[1][\proofname]{%
  \oldproof[\upshape \bfseries #1]%
}%
{\oldendproof}
\makeatletter
\def\@makechapterhead#1{%
  \vspace*{50\p@}%
  {\parindent \z@ \raggedright \normalfont
    \interlinepenalty\@M
    \Huge\bfseries  \thechapter.\quad #1\par\nobreak
    \vskip 40\p@
  }}
\makeatother

\newcommand{\reqnomode}{\tagsleft@false}
\def \d{\,{\rm d}}

\def\supp{\,{\rm supp \,}}

\allowdisplaybreaks
\makeatletter
\DeclareRobustCommand*{\bfseries}{%
  \not@math@alphabet\bfseries\mathbf
  \fontseries\bfdefault\selectfont
  \boldmath
}

\makeatother

\newlength{\defbaselineskip}
\newcommand\eps\varepsilon
\def\mean#1{\mathchoice%
          {\mathop{\kern 0.2em\vrule width 0.6em height 0.69678ex depth -0.58065ex
                  \kern -0.8em \intop}\nolimits_{\kern -0.4em#1}}%
          {\mathop{\kern 0.1em\vrule width 0.5em height 0.69678ex depth -0.60387ex
                  \kern -0.6em \intop}\nolimits_{#1}}%
          {\mathop{\kern 0.1em\vrule width 0.5em height 0.69678ex
              depth -0.60387ex
                  \kern -0.6em \intop}\nolimits_{#1}}%
          {\mathop{\kern 0.1em\vrule width 0.5em height 0.69678ex depth -0.60387ex
                  \kern -0.6em \intop}\nolimits_{#1}}}

\allowdisplaybreaks[4]

\newcommand\blfootnote[1]{%
  \begingroup
  \renewcommand\thefootnote{}\footnote{#1}%
  \addtocounter{footnote}{-1}%
  \endgroup
}

\def\eqn#1$$#2$${\begin{equation}\label#1#2\end{equation}}
\newcommand\R{\mathbb{R}}

\newcommand{\F}{\mathscr F}

\def \tp{\textup}
\def \p{\partial}
\def \e{\varepsilon}

\def \D{\mathrm D}

\def \LL {\mathrm L}
\def \WW{\mathrm W}

\newcommand\restr[2]{{% we make the whole thing an ordinary symbol
  \left.\kern-\nulldelimiterspace % automatically resize the bar with \right
  #1 % the function
  \vphantom{|} % pretend it's a little taller at normal size
  \right|_{#2} % this is the delimiter
  }
}

\def\XXint#1#2#3{\setbox0=\hbox{$#1{#2#3}{\int}$}\vcenter{\hbox{$#2#3$}}\kern-.5\wd0}

\AtBeginDocument{

}

\mathtoolsset{showonlyrefs}

\usepackage{import}
\usepackage{xifthen}
\usepackage{pdfpages}
\usepackage{transparent}

\title{On global absence of Lavrentiev gap for functionals with $(p,q)$-growth}
\author[1]{Lukas Koch}
  
  \affil[1]{\small MPI for Mathematics in the Sciences, Inselstrasse 22, 04177 Leipzig, Germany
\protect \\
  {\tt{kochl@mis.mpg.de}}
  \vspace{1em} \ }

\usepackage{etoolbox}
\makeatletter
\patchcmd{\@adjustvertspacing}
  {\jot\baselineskip \divide\jot 4}
  {\jot=.5\baselineskip}
  {}{}
\makeatother

\makeatother
\begin{document}
\maketitle

\begin{abstract}
We prove that for convex vectorial functionals with $(p,q)$-growth the Lavrentiev phenomenon does not occur up to the boundary when $(p,q)$ are suitably restricted. Under minimal assumptions on the boundary data, we obtain results for autonomous and non-autonomous functionals, under natural, controlled and controlled duality growth bounds. 
\end{abstract}

\blfootnote{
\emph{2010 Mathematics Subject Classification:} 35J60, 35J20\\ 
\emph{Keywords.} Lavrentiev phenomenon, nonstandard growth, vectorial convex functional\\

}

\section{Introduction and results}
We consider minimisation problems of the form
\begin{align}\label{eq:min}
\min_{u\in g+\WW^{1,p}_0(\Omega,\R^m)}\F(u) \quad\text{ where }\quad \F(u)= \int_\Omega F(x,\D u)\d x.
\end{align}
Here $\Omega\subset \R^n$ is a Lipschitz domain, $g\in \WW^{1,q}(\Omega,\R^m)$ is the Dirichlet boundary condition and $F$ is an integrand satisfying $(p,q$)-growth conditions in the gradient variable.

It is well known that the Lavrentiev phenomenon provides a fundamental obstruction to the regularity theory regarding \eqref{eq:min}. The Lavrentiev phenomenon describes the fact that it may occur that
$$
\min_{u\in g+\WW^{1,p}_0(\Omega,\R^m)}\F(u) < \min_{u\in g+\WW^{1,q}_0(\Omega,\R^m)} \F(u).
$$
This observation was first made in \cite{Lavrentiev1926}. With regards to \eqref{eq:min}, the theory was further developed in \cite{Zhikov1987}, \cite{Zhikov1993} and \cite{Zhikov1995}. In this paper, we adopt the viewpoint and terminology of \cite{Buttazo1992} and view the Lavrentiev phenomenon through the so-called Lavrentiev gap. 

Suppose $X$ is a topological space of weakly differentiable functions and $Y\subset X$. Introduce
\begin{align*}
\overline \F_X = \sup\{\,\mathscr G\colon X\to[0,\infty]: \mathscr G \text{ slsc }, \mathscr G\leq \F \text{ on } X\,\}\\
\overline \F_{Y} = \sup\{\,\mathscr G\colon X\to[0,\infty]: \mathscr G \text{ slsc }, \mathscr G\leq \F \text{ on } Y\,\}\nonumber.
\end{align*}
The Lavrentiev gap functional is then defined for $u\in X$ as
\begin{align*}
\mathscr L(u,X,Y)=\begin{cases}
		\overline \F_{Y}(u)-\overline \F_X(u)  &\text{ if } \overline \F_X(u)<\infty\\
		0 &\text{ else}.
	\end{cases}
\end{align*}
Note that $L(u,X,Y)\geq 0$ and that $L(u,X,Y)>0$ for some $u\in X$ when the Lavrentiev phenomenon occurs. However, in general, it could be that $L(u,X,Y)>0$ for some $u\in X$, but the Lavrentiev phenomenon does not occur. There is an extensive literature on the Lavrientiev phenomenon and gap functional in this abstract set-up, an overview of which can be found in \cite{Buttazo1995}, \cite{Foss2001} to which we also refer for further references.

We remark that the Lavrentiev phenomenon is also of interest in nonlinear elasticity \cite{Foss2003}. Further, an important direction of study of the Lavrentiev gap functional is to obtain measure representations of $\overline \F_Y$, in particular in the case when $X=\WW^{1,p}(\Omega,\R^m)$ and $Y = \WW^{1,q}_{\tp{loc}}(\Omega,\R^m)$. We do not pursue this direction further here, but refer to \cite{Fonseca1997,Acerbi2003} for results and further references.

In order to compare our assumptions and results to the literature, we state them precisely. Let $1<p\leq q$. Suppose $F(\cdot,z)$ is continuous for any $z\in \R^{n\times m}$ and $F(x,\cdot)$ is $C^1(\Omega)$ for almost every $x\in \Omega$. We will always assume a strict convexity assumption on $F$: There is $\nu>0$ and $\mu\in[0,1]$ such that for almost every $x\in\Omega$ and every $z\in \R^{n\times m}$,
\begin{gather}\label{def:ellipticity}
\nu(\mu^2+\lvert z\rvert^2+\lvert w\rvert^2)^\frac{p-2}{2}\leq\frac{F(x,z)-F(x,w)-\langle \partial_z F(x,w),z-w\rangle}{\lvert z-w\rvert^2}\tag{H1}.
\end{gather}
We quantify the growth of $F$ through one of the following growth conditions: There is $\Lambda>0$ such that for almost every $x\in\Omega$ and every $w,z\in \R^{n\times m}$,
\begin{gather}
\label{def:naturalGrowth}
\lvert F(x,z)\rvert\leq \Lambda(1+\lvert z\rvert^2)^\frac q 2\tag{H2}\\
\label{def:controlledGrowth}
\frac{F(x,z)-F(x,w)-\langle \partial_z F(x,w),z-w\rangle}{\lvert z-w\rvert^2}\leq \Lambda(1+\lvert z\rvert^2+\lvert w\rvert^2)^\frac{q-2} 2\tag{H3}\\
\label{def:controlledDualityGrowth}
\frac{F(x,z)-F(x,w)-\langle \partial_z F(x,w),z-w\rangle}{\lvert z-w\rvert^2}\leq \Lambda(1+\lvert \p_z F(z)\rvert^2+\lvert\p_z F(w)\rvert^2)^\frac{q-2} {2(q-1)}\tag{H4}.
\end{gather}
We remark that under these assumptions the existence of a solution to \eqref{eq:min} follows from the direct method.
\eqref{def:naturalGrowth}--\eqref{def:controlledDualityGrowth} are increasingly strong assumptions, that are referred to as natural growth, controlled growth and controlled duality growth assumptions, respectively.

Finally, we assume that $F(\cdot,z)$ is H\"older-continuous, that is for almost every $x,y\in\Omega$ and every $z\in\R^{n\times m}$,
\begin{gather}
\label{def:xHolder} \lvert F(x,z)-F(y,z)\rvert\leq \Lambda \lvert x-y\rvert^\alpha\left(1+ \lvert z\rvert^2\right)^\frac q 2.\tag{H5}
\end{gather}

In the case of controlled growth conditions, we need to make the strong assumption that $F$ is doubling, that is there is $s_0>0$ such that for $s\in[1,s_0)$, 
\begin{align}\label{def:doubling}
F(s z)\lesssim 1+F(z).\tag{H6}
\end{align}
It would be desirable to remove this assumption.

We remark that a consequence of any of \eqref{def:naturalGrowth}--\eqref{def:controlledDualityGrowth} in combination with the convexity of $F(x,\cdot)$ expressed in \eqref{def:ellipticity} is the bound:
\begin{align}\label{eq:diffF}
\lvert F(x,z)-F(x,w)\rvert\lesssim \lvert z-w\rvert(1+\lvert z\rvert+\lvert w\rvert)^{q-1}
\end{align}
for almost every $x\in\Omega$ and every $z,w\in \R^{n\times m}$.

For technical reasons, we additionally require the following condition, which we will discuss in more detail soon: there is $\e_0>0$ such that for any $\e\in(0,\e_0)$ and $x\in\Omega$ there is $\hat y\in \overline{B_\e(x)\cap\Omega}$ 
such that
\begin{align}\label{def:xCondition}
F(\hat y,z)\leq F(y,z) \qquad\forall y\in \overline{B_\e(x)\cap\Omega},\hspace{0.5cm} z\in \R^{n\times m}.\tag{H6}
\end{align}

Introduce
\begin{align}\label{eq:relaxedFunc}
\overline \F(u) = \inf\left\{\liminf \F(u_j)\colon (u_j)\subset Y, \, u_j\rightharpoonup u \text{ weakly in } X\right\},
\end{align}
where $X = g+\WW^{1,p}_0(\Omega,\R^m)$ and $Y = g+\WW^{1,q}_0(\Omega,\R^m)$. Since $F(x,z)$ is convex due to \eqref{def:ellipticity}, with this choice of $X$ and $Y$, standard methods show that $\overline\F_X(\cdot)=\F(\cdot)$, see \cite[Chapter 4]{Giusti2003}. Further $\overline \F_Y(\cdot)=\overline \F(\cdot)$.

We then prove the following result:
\begin{theorem}\label{thm:noLavrentiev}
Let $1< p\leq q$ and suppose $\Omega\subset \R^n$ is a Lipschitz domain. Assume $g\in \WW^{1,q}(\Omega,\R^m)$ and take $\alpha\in(0,1)$. Let $u\in g+\WW^{1,p}_0(\Omega,\R^m)$.
Suppose $F$ satisfies \eqref{def:ellipticity} and one of the following conditions hold:
\begin{enumerate}
\item $q< \frac{(n+\alpha)p} n$ and $F$ satisfies \eqref{def:naturalGrowth}, \eqref{def:xHolder} and \eqref{def:xCondition},\label{eq:nonautonomous}
\item $q< \min\left(p+1,\frac{np}{n-1}\right)$ and $F\equiv F(z)$ satisfies \eqref{def:naturalGrowth},\label{eq:natural}
\item $p\leq q<p+\max\left(1,\frac p n\right)$, $u\in\LL^\infty(\Omega)$ and $F\equiv F(z)$ satisfies \eqref{def:naturalGrowth},\label{eq:Linftynatural}
\item  $p\leq q<p+\max\left(1,\frac p n\right)$, $m=1$ and $F\equiv F(z)$ satisfies \eqref{def:naturalGrowth},\label{eq:scalarnatural}
\end{enumerate}
Then
$$
\F(u)= \overline\F(u).
$$
If $F$ satisfies \eqref{def:ellipticity}, $u\in g+\WW^{1,p}_0(\Omega)$ solves \eqref{eq:min} and one of the following conditions holds:
\begin{enumerate}\setcounter{enumi}{4}
\item $2\leq p\leq q< \min\left(p+2,p\left(1+\frac 2 {n-1}\right)\right)$ and $F\equiv F(z)$ satisfies \eqref{def:controlledGrowth}, \eqref{def:doubling}\label{eq:controlled}
\item $2\leq p\leq q< \frac{np}{n-p}$ and $F\equiv F(z)$ satisfies \eqref{def:controlledDualityGrowth},\label{eq:controlledDuality}
\item $2<p\leq q<p+\max\left(2,\frac {2p} n\right)$, $u\in \LL^\infty(\Omega)$ and $F\equiv F(z)$ satisfies \eqref{def:controlledGrowth},\label{eq:Linftycontrolled}
\item  $2\leq p\leq q<p+\max\left(2,\frac {2p} n\right)$, $m=1$ and $F\equiv F(z)$ satisfies \eqref{def:controlledGrowth}, \eqref{def:doubling}\label{eq:scalarcontrolled},
\end{enumerate}
then 
\begin{align*}
\F(u) = \inf_{v\in g+\WW^{1,q}_0(\Omega)} \F(v).
\end{align*}
\end{theorem}

%\begin{remark}
%Cases \ref{eq:Linftynatural}, \ref{eq:scalarnatural}, \ref{eq:Linftycontrolled} and \ref{eq:scalarcontrolled} also hold for some $x$-dependent integrands. We discuss this question in more detail in Remark \ref{rem:xDep}.
%\end{remark}

Integrands with $(p,q)$-growth, such as the ones we consider in this paper, have been studied since the seminal papers \cite{Marcellini1989}, \cite{Marcellini1991}. There is by now an extensive literature regarding the regularity theory. We don't aim to give a complete overview of this theory here, but refer to \cite{Mingione2006,Mingione2021} for a good overview and further references. We remark that the first step in the regularity theory is to prove that minimisers of $\overline \F$ are $\WW^{1,q}$-regular. As soon as the Lavrentiev gap is excluded, this proves that minimisers of $\F$ are $\WW^{1,q}$-regular. Hence, it is useful to compare the range of $(p,q)$ in Theorem \ref{thm:noLavrentiev} to the range in available (local) $\WW^{1,q}$- regularity results. 

Under the assumptions of \ref{eq:nonautonomous}, minimisers of $\overline \F$ are $\WW^{1,q}_{\tp{loc}}$-regular and this assumption is sharp \cite{Esposito2004,Esposito2019}, see also \cite{Diening2020} regarding the sharpness. If $g\in \WW^{1+\alpha,q}(\Omega,\R^m)$ and $\Omega$ is Lipschitz, minimisers of $\overline \F$ are even $\WW^{1,q}$ regular \cite{Koch2020}.  In the setting of \ref{eq:Linftynatural} (assuming only $q<p+1$) and \ref{eq:Linftycontrolled} (assuming only $q<p+2$), $\WW^{1,q}_{\tp{loc}}$-regularity is due to \cite{Carozza2011}. Regularity in the range given in \ref{eq:Linftynatural} and \ref{eq:Linftycontrolled}, respectively, follows from combining these results with those in \cite{Esposito2004,Esposito2019}, that is corresponding to the set-up of \ref{eq:nonautonomous} with the choice $\alpha =1$. In particular, we note that in these settings we show that there is no Lavrentiev gap precisely in the regime where $\WW^{1,q}_{\tp{loc}}$-regularity is known. We remark that if the a-priori assumption on $u$ is strengthened beyond $u\in \LL^\infty$, $\WW^{1,q}_{\tp{loc}}$-regularity of minimisers can be obtained under weaker assumptions than those of \ref{eq:Linftynatural} or \ref{eq:Linftycontrolled}, c.f. \cite{deFilippis2021c}, but we do not pursue this direction here.

In the setting of \ref{eq:natural}, $\WW^{1,q}_{\tp{loc}}$-regularity of minimisers of $\overline \F$ if $1<p<\frac{np}{n-1}$ is due to \cite{Carozza2013}. If $g\in \WW^{2,q}(\Omega,\R^m)$, $\Omega$ is Lipschitz and $q<\min\left(p+1,\frac{np}{n-1}\right)$, relaxed minimisers are even $\WW^{1,q}$ regular \cite{Koch2021a}. Thus, we recover this range of $(p,q)$ in \ref{eq:natural}. $\WW^{1,q}_{\tp{loc}}$ regularity in the setting of \ref{eq:controlled} with $q<p\left(1+\min\left(\frac 2 {n-1},1\right)\right)$ is due to \cite{Schaeffner2020}. Again, we note that we recover this range of $(p,q)$ if $n\geq 3$ and $p\leq n-1$ in \ref{eq:controlled}. In the scalar case $m=1$ stronger results are known. In particular, if $\lvert z\rvert^p\lesssim F(z)\lesssim (1+\lvert z\rvert^q)$ and $F$ is convex, it suffices to assume $\frac q p\leq 1+\frac q {n-1}$ in order to ensure local boundedness of minimisers of $\overline \F$, \cite{Hirsch2020}. This restriction is sharp \cite{Marcellini1991,Hong1992}. Finally, we remark that the case of controlled-duality growth is treated in the upcoming \cite{DeFilippis2020}, assuming $q\leq\frac{np}{n-2}$. 

In order to explain our proof, it is instructive to first consider the interior case when $F\equiv F(z)$ is autonomous: Let $\{\rho_\e\}$ be a standard family of mollifiers. Given $u\in g+\WW^{1,p}_0(\Omega,\R^m)$ and $\omega\Subset\Omega$, for $\e>0$ sufficiently small, $u_\e = u\star \rho_\e\in \WW^{1,q}_{\tp{loc}}(\Omega,\R^m)$. Moreover, due to Jensens' inequality and the convexity of $F$, i.e. \eqref{def:ellipticity},
\begin{align*}
\int_\omega F(\D u_\e)\d x\leq \int_\omega F(\D u)\star \rho_\e \d x \xrightarrow{\e\to 0} \int_\omega F(\D u)\d x.
\end{align*}
Thus, in the local convex autonomous case, there is no Lavrentiev gap.

From this calculation it is clear, that there are two substantial difficulties in proving Theorem \ref{thm:noLavrentiev}. We need to adapt the mollification near the boundary and the application of Jensen's inequality needs to be justified in the non-autonomous setting. To our knowledge, Theorem \ref{thm:noLavrentiev} is the first result excluding the Lavrentiev phenomenon globally in the vectorial setting without assuming structure conditions on $F$ going beyond $p$-ellipticity and $q$-growth.

There are various approaches towards dealing with the non-autonomous setting. Condition \eqref{def:xCondition} holds for many of the commmonly considered examples, see \cite{Esposito2004,Koch2020}. It is very similar to \cite[Assumption 2.3]{Zhikov1995} and has also appeared in the context of lower semi-continuity in \cite{Acerbi1994a}. In the context of functionals with Orlicz growth, similar assumptions have been used, see for example \cite{Bulicek2022}, \cite{Harjulehto2018}.
Moreover, recently in \cite{deFilippis2022} the Lavrentiev gap has been excluded for non-autonomous integrands in the setting of Theorem \ref{thm:noLavrentiev}.\ref{eq:nonautonomous} with \eqref{def:xCondition} replaced by the assumption that the integrand $F$ can be approximated from below by integrands $F_k$ satisfying
$$
\lvert \p_z F_k(x,z)-\p_z F_k(x,z)\rvert\lesssim c_k\lvert x-y\rvert^\alpha(1+\lvert z\rvert^{p-1}).
$$
Moreover, for special cases, such as the double-phase functional ${F(x,z)=\lvert z\rvert^p+a(x)\lvert z\rvert^q}$, the equality case $q=\frac{(n+\alpha)p} n$ has been settled \cite{Balci2021}. However, at the moment there seems to be no condition available that covers all cases in which the (local) Lavrentiev phenomenon is known not to occur. Finally, we would like to contrast our results, in particular the cases \ref{eq:Linftynatural} and \ref{eq:scalarnatural} of Theorem \ref{thm:noLavrentiev} with those obtained in \cite{Bulicek2022}. The results in \cite{Bulicek2022} concern integrands with Orlicz-growth, a direction we do not pursue here. When restricted to integrands with polynomial $(p,q)$-growth, we recover the results of \cite{Bulicek2022} for a larger class of integrands $F$. Further, we remark that the results of \cite{Bulicek2022} have been extended to anisotropic functionals in \cite{Borowski2022}.

Regarding the mollification near the boundary, in this paper, we expand a technique already utilised in \cite{Koch2020}, \cite{Koch2021a}. We emphasize that in the analysis of \cite{Koch2020}, \cite{Koch2021a}, the Lavrentiev phenomenon was excluded, but no result was proven regarding the Lavrentiev gap. Moreover, only the setting of cases \ref{eq:nonautonomous} and \ref{eq:natural} in Theorem \ref{thm:noLavrentiev} was studied and stronger assumptions on the regularity of $g$ were imposed. The key idea is to use a two-parameter mollification $u\star \phi_{\e \delta(x)}$ where $\delta(x)\sim d(x,\p\Omega)$. This idea will be implemented using a Whitney-Besicovitch covering of $\Omega$. Similar ideas have been used in \cite{Fonseca2005a} in relation with studying measure representations of $\overline \F$ and more recently in \cite{Dyda2022} in order to prove density of smooth functions for weighted fractional Sobolev spaces on open sets.

The rest of this paper is structured as follows. In Section \ref{sec:prelim}, we introduce our notation and collect a number of preliminary results. In Section \ref{sec:construction}, we construct for ${u\in g+\WW^{1,p}_0(\Omega,\R^m)}$ a sequence of regular approximations $u_\e\in g+\WW^{1,q}_0(\Omega,\R^m)$. In Section \ref{sec:energies}, we prove under various assumptions that $\F(u_\e)\to \F(u)$. Finally, we prove Theorem \ref{thm:noLavrentiev} in Section \ref{sec:theoremProof}.

\section{Preliminaries}\label{sec:prelim}
Throughout $\Omega$ is a Lipschitz domain of $\R^n$. Denote by $\overline\Omega$ the closure of $\Omega$. Given $\e>0$,  set $\Omega_\e = \{x\in\Omega\colon d(x,\p\Omega)<\e\}$. For two sets $A,B\subset\R^n$ we denote $A+B =\{a+b\colon x\in A, b\in B\}$.
 We write $B_r(x)$ for the usual open Euclidean ball of radius $r$ in $\R^n$. For readability, we adopt the non-standard notation $s B_r(x) = B_{sr}(x)$ for $s>0$.
$\lvert\cdot \rvert$ denotes the Euclidean norm of a vector in $\R^n$ and likewise the Euclidean norm of a matrix $A\in \R^{n\times n}$. 

If $p\in[1,\infty]$ denote by $p^\prime=\frac{p}{p-1}$ its H\"older conjugate.
The symbols $a \sim b$ and $a\lesssim b$ mean that there exists some constant $C>0$, depending only on $n,m,p,q,\Omega,\mu,\lambda$ and $\Lambda$, and independent of $a$ and $b$ such that $C^{-1} a \leq b \leq C a$ and $a\leq C b$, respectively. 

Let $1<p\leq \infty$. We denote by $\LL^p(\Omega)=\LL^p(\Omega,\R^m)$ the usual Lebesgue space on $\Omega$. Since we do not make a regularity assumption on $\Omega$, we need to be precise in our definition of Sobolev spaces. Denote by $C^\infty(\overline\Omega)=C^\infty(\overline\Omega,\R^m)$, the set of smooth $\R^m$-valued functions in $\overline\Omega$ and by $C_c^\infty(\Omega)=C_c^\infty(\Omega,\R^m)$, the set of smooth $\R^m$-valued functions compactly supported in $\Omega$. Then we consider $\WW^{1,p}(\Omega,\R^m)=\WW^{1,p}(\Omega)$ to be the closure of $C^\infty(\overline\Omega)$ in $\Omega$ with respect to $\|\cdot\|_{\WW^{1,p}(\Omega)}$. Here $\|\cdot\|_{\WW^{1,p}(\Omega)}$ denotes the usual Sobolev norm. $\WW^{1,p}_0(\Omega)=\WW^{1,p}_0(\Omega,\R^m)$ denotes the closure of $C_c^\infty(\Omega)$ functions with respect to $\|\cdot\|_{\WW^{1,p}(\Omega)}$. For the theory of these spaces we refer to \cite{Grisvard1992}.

Whitney-Besicovitch coverings will be a key tool in our construction. They combine properties of Whitney and Besicovitch coverings and were introduced in \cite{Kislyakov2005}. A nice presentation of the theory is given in \cite{Kislyakov2013}. To be precise:
\begin{definition}
 A family of balls $\{\,B_i\,\}_{i\in I}$ is called a Whitney-Besicovitch-covering (WB-covering) of $\Omega$ if there is a triple $(\delta, M, \varepsilon)$ of positive numbers such that
 \begin{align}\label{def:WBcoverExtension}
  & \bigcup_{i\in I} B_i = \bigcup_{i\in I} (1+\delta)B_i = \Omega\\\label{def:WBcoverMultiplicity}
  & \sum_{i\in I} \chi_{(1+\delta)B_i}\leq M\\
  & B_i \cap B_j \neq \emptyset \Rightarrow \lvert B_i \cap B_j\rvert \geq \varepsilon \max(\lvert B_i\rvert,\lvert B_j\rvert).  
 \end{align}
 \end{definition}

For any open set $\Omega\subset \R^n$, a WB-covering of $\Omega$ exists due to \cite[Theorem 3.1]{Kislyakov2013}. We note that we can furthermore ensure that there is $C>0$ such that, for all $i\in I$,
\begin{align}\label{eq:compareBoundary}
\frac 1 C d(B_i,\p\Omega)\leq \lvert B_i\rvert^\frac 1 n\leq C d(B_i,\p\Omega).
\end{align}
Finally, a consequence of \eqref{def:WBcoverExtension} and \eqref{eq:compareBoundary} is that
\begin{align}\label{eq:overlappingCubes}
(1+\delta/2)B_i\cap (1+\delta/2)B_j\neq\emptyset \Rightarrow \lvert B_i\rvert\sim \lvert B_j\rvert.
\end{align}

We will require a partition of unity related to a WB-covering.
\begin{theorem}\label{thm:partition}
Let $\Omega\subset \R^n$ be open. Suppose the balls $\{B_i\}_{i\in I}$ form a WB-covering of $\Omega$. Let $B_i$ have radius $r_i$. Suppose $\{\phi_i\}_{i\in I}\subset C_c^\infty(\Omega)$ is such that
\begin{align}\label{eq:phiAss}
\supp \phi_i \subset (1+\delta/2)B_i,\quad 0\leq \phi_i \, \text{ and } \, \phi_i = 1 \text{ on } B_i
\end{align}
for all $i\in I$.
Then there is a family  family $\{\,\psi_i\,\}_{i\in I}$ of infinitely differentiable functions that form a partition of unity on $\Omega$ with the following properties:
 \begin{align}\label{eq:partitionOfUnityProp}
  &\supp (\psi_i) \subset (1+\delta/2) B_i, \qquad 0\leq \psi_i\leq 1 \quad\text{ and } \quad\, \psi_i(x) \geq \frac{1}{M} \text{ for } x\in B_i
 \end{align}
 for all $i\in I$. Further,
\begin{align*}
\lvert \D\psi_i\rvert\lesssim \sum_{\{j\colon(1+\delta/2)B_i\cap (1+\delta/2)B_j)\neq\emptyset\}} \lvert \D\phi_j\rvert.
\end{align*} 
 Moreover, it is possible to choose $\{\phi_i\}_{i\in I}$ such that
 \begin{align}\label{eq:partitionOfUnityDeriv}
   \lvert D\psi_i\rvert \leq c r_i^{-1}.
   \end{align}
\end{theorem}
\begin{proof}
Suppose $\{\phi_i\}$ satisfies \eqref{eq:phiAss}. Note that then
$$
\psi_i = \frac{\phi_i}{\sum \phi_i}
$$
is a smooth partition of unity of $\Omega$ satisfying \eqref{eq:partitionOfUnityProp} due to \eqref{def:WBcoverMultiplicity}. We further note that
$$
\D\psi_i = \frac{\D \phi_i}{\sum_j \phi_j}+\frac{\phi_i}{(\sum\phi_j)^2}\sum_j \D\phi_j.
$$
In particular,
\begin{align}\label{eq:pouDeriv}
\lvert\D\psi_i\rvert\lesssim \sum_{\{j\colon (1+\delta/2)B_i\cap (1+\delta/2)B_j\neq\emptyset\}} \lvert\D\phi_j\rvert
\end{align}
as $\sum_j \phi_j\geq 1$.

Finally, due to \eqref{eq:compareBoundary}, it is possible to choose $\{\phi_i\}$ such that $\lvert \D \phi_i\rvert\lesssim \lvert B_i\rvert^{-\frac 1 n}$. \eqref{eq:pouDeriv} and \eqref{eq:overlappingCubes} then give \eqref{eq:partitionOfUnityDeriv}.
\end{proof}

We further record a number of well-known estimates concerning rate of convergence of mollifications. We fix a family $\{\,\rho_\e\,\}$ of radially symmetric, non-negative mollifiers of unitary mass. We denote convolution with $\rho_\e$ as
\begin{align*}
u\star\rho_\e(x)=\int_{\R^n} u(y)\phi_\e(x-y)\d y.
\end{align*}

\begin{lemma}\label{lem:mollification}
Let $1<p<\infty$ and $0<r\leq 1$.
Suppose $u\colon \R^n\to \R^m$. Let $B=B_r(x_0)\subset \R^n$. For $\e>0$ denote $u_\e= u\star \rho_\e$. Then, the following estimates hold (with implicit constants independent of $r$):
\begin{enumerate}
\item $\|u_\e\|_{\LL^p(B)}\lesssim \|u\|_{\LL^p(B+\e B_1(0))}$,
\item $\|u_\e\|_{\LL^\infty(B)}\lesssim \|u\|_{\LL^\infty(B+\e B_1(0))}$,
\item $\|u_\e\|_{\LL^\infty(B)}\lesssim \e^{-\frac n p}\|u\|_{\LL^p(B+\e B_1(0))}$,
\item $\|u_\e-u\|_{\LL^p(B)}\lesssim \e \|u\|_{\WW^{1,p}(B+\e B_1(0))}$,						
\end{enumerate}
Moreover,
\begin{enumerate}
\item	for $p\leq q\leq \frac{np}{n-p}$, $
	\|u_\e-u\|_{\LL^q(B)}\lesssim \e^{1+n\left(\frac 1 q-\frac 1 p\right)}\|u\|_{\WW^{1,p}(B+\e B_1(0))}^\frac p q,$\\
\item for $p\leq q<\infty$, $\|u_\e-u\|_{\LL^q(B)}\lesssim \e^\frac p q \|u\|_{\WW^{1,p}(B+\e B_1(0))}^\frac p q\|u\|_{\LL^\infty(B+\e B_1(0))}^{1-\frac p q}$.
\end{enumerate}
\end{lemma}
\begin{proof}
The estimates are standard. Proofs for the first part can be found in \cite{evans}. The moreover part follows by interpolation from the first part.
\end{proof}

Next, we present two key technical tools. First, we require a lemma that essentially allows us to work on spheres with regards to certain estimates. This is an observation, originally proven in \cite{Schaeffner2020}. We state a slightly modified version as in \cite{Koch2021a}:
\begin{lemma}\label{lem:schaffner}
Fix $n\geq 2$ and let $t_1>t_2\geq1$. For given $0<\rho<\sigma<\infty$ with $\sigma-\rho<1$ and $w\in L^1(B_\sigma)$, consider
\begin{align*}
J(\rho,\sigma,w) = \inf \left\{\int_{B_\sigma} \lvert w\rvert (\lvert\D \phi\rvert^{t_1}+\lvert\D \phi\rvert^{t_2})\d x\colon \phi\in C^1_0(B_\sigma),\, \phi \geq 0,\, \phi =1 \text{ in } B_\rho\right\}.
\end{align*}
Then, for every $\delta\in(0,1)$, we have
$$
J(\rho,\sigma,w)\leq (\sigma-\rho)^{-{t_1}-1/\delta}\left(\int_\rho^\sigma \left(\int_{\p B_r} \lvert w\rvert\d \sigma\right)^\delta \d r\right)^\frac {1} {\delta}.
$$
\end{lemma}

Second, we recall a lemma from \cite{Esposito2019}, which will be crucial for dealing with non-autonomous integrands:
\begin{lemma}\label{lem:changeOfXAlt}
Assume $1<p\leq q< \frac{(n+\alpha)p}{n}$. Suppose $F$ satisfies \eqref{def:ellipticity}, \eqref{def:naturalGrowth}, \eqref{def:xHolder} and \eqref{def:xCondition}.
Then for $x\in\Omega$ and $\e\leq \min(\e_0,d(x,\p\Omega))$,
\begin{align*}
F(x,Du(\cdot)\star \phi_\e(x))\lesssim 1 + \big(F(\cdot,Du(\cdot))\star \phi_\e\big)(x).
\end{align*}
Moreover, if $\lvert z\rvert \leq C \e^{-\frac n p}$, it holds
\begin{align*}
\sup_{y\in B_\e(x)}F(y,z)\leq 1+\inf_{y\in B_\e(x)} F(y,z)
\end{align*}
\end{lemma}

Finally, we recall the following lemma, which is a combination of \cite[Lemma 10]{Koch2020} and \cite[Theorem 5]{Koch2022}.
\begin{lemma}\label{lem:EulerLagrange}
Let $1<p$.
Suppose $F$ satisfies \eqref{def:ellipticity} and $F(x,z)\lesssim 1+\lvert z\rvert^q$. Then, if $u\in g+\WW^{1,p}_0(\Omega)$ solves \eqref{eq:min},
\begin{align*}
\int_\Omega \p_z F(x,\D u)\cdot \D\phi\d x = 0 \quad \text{ for all } \phi\in \WW^{1,q}_0(\Omega).
\end{align*}
and
\begin{align*}
\int_\Omega \p_z F(x,\D u)\cdot(\D u-\D g)\d x\leq 0.
\end{align*}
Moreover, if $q\leq \frac{np}{n-p}$,
\begin{align*}
\p_z F(x,\D u)\in \LL^{q^\prime}(\Omega).
\end{align*}

\end{lemma}

\section{A simplification}
We claim that in order to prove Theorem \ref{thm:noLavrentiev}, we may make the following additional assumption:
\begin{claim}
Under the assumptions of Theorem \ref{thm:noLavrentiev}, it suffices to construct $(u_j)\subset g+\WW^{1,p}_0(\Omega)\cap \WW^{1,q}_{\tp{loc}}(\Omega)$ such that $u_j\to u$ in $\WW^{1,p}(\Omega)$.
\end{claim}

We first focus on the case where $f\equiv f(z)$. Since $\Omega$ is a Lipschitz domain, there exists a finite family of domains $x_i+\Omega_i$, $i=1,\ldots m$ such that $\Omega_i$ is strongly star-shaped with respect to $0$ and $\Omega = \cup (x+i+\Omega_i)$. We may also ensure that there exist $\omega_i\Subset\Omega_i$ with respect to $\Omega$, strongly star-shaped with respect to $x_i$ and such that still $\Omega = \cup (x_i+\omega_i)$. Extend $g$ to a $\WW^{1,q}$-function on $\R^n$ and extend $u$ by $g$ outside of $\Omega$. We want to introduce an approximation to $u$ on $x_i+\Omega_i$. After a translation, we may assume $x_i=0$. Then consider for $s>1$ and $t=t(s)<1$ to be chosen at a later point,
\begin{align*}
u^s_i(x)= t\left(\frac 1 s(u-g)(s x)+ g(x)\right).
\end{align*}
Clearly $u_i^s\to u$ in $\WW^{1,p}(\Omega_i)$ and $u^s_i = g$ on $\p\Omega$. Further $u^s_i$ is $\WW^{1,q}$ in an open neighbourhood of $(\p\omega \cap \p\Omega)$. Moreover, using the convexity of $F$,
\begin{align*}
\int_{\Omega_i} F(\D u_s^i)\d x =&\int_{\Omega_i} F\left(t\D(u-g)(s x)+\D g(s x)-t\D g(sx)+t\D g(x)\right)\d x \\
\leq& t\int_{\Omega_i} F(\D(u-g)(sx)+\D g(s x))\d x + (1-t)\int_{\Omega_i} F\left(\frac t {1-t}(\D g(x)-\D g(sx)\right)\d x.
\end{align*}
On the one hand,
\begin{align*}
\int_{\Omega_i} F(\D u(sx))\d x=& \int_{s\Omega_i\cap \Omega} F(\D u)s^{-d}\d x+ \int_{s\Omega_i\setminus(\Omega_i\cap \Omega)} F(\D g)s^{-d}\d x \to \int_{\Omega_i} F(\D u)\d x.
\end{align*}
On the other hand,
\begin{align*}
(1-t)\int_{\Omega_i} F\left(\frac t {1-t} (\D g(x)-\D g(sx)\right)\d x \lesssim \frac{t^q} {(1-t)^{q-1}} \int_{\Omega_i} \lvert \D g(x)-\D g(sx)\rvert^q\d x.
\end{align*}
Note that this tends to $0$ as $s\to 1$. In particular, we may choose $t(s)$ such that by a version of the dominated convergence theorem,
\begin{align*}
\int_{\Omega_i} F(\D u_s^i)\d x \to \int_{\Omega_i} F(\D u)\d x.
\end{align*}
Let $\{\eta_i\}$ be a smooth partition of unity such that $\mathrm{supp}\;\eta_i\subset x_i+\omega_i$. Then define,
\begin{align*}
u^{s}(x) = \eta_i u_i^s(x).
\end{align*}
Note that clearly $u^{s}(x)\to u$ in $\WW^{1,p}(\Omega)$ as $i\to\infty$, $u^{s}\in g+\WW^{1,p}_0(\Omega)$ and moreover $u^s$ is $\WW^{1,q}$ in an open neighbourhood of $\p\Omega$. Further, ensuring also that $\sup_{x\in \supp\;\eta_i} \frac{\lvert \D\eta\rvert}{(1-\eta_i)^{q-1}}\leq C<\infty$,
\begin{align*}
\int_\Omega F(\D u^s(x))\d x =& \sum_i \int_{\Omega_i} F\left(\eta_i \D u_i^s(x)+ u_i^s \otimes\D\eta_i\right)\\
\leq& \sum_i \int_{\supp\;\eta_i} \eta_i F\left(\D u_i^s(x)\right)\d x + \sum_i \int_{\supp\;\eta_i} (1-\eta_i) F\left(\frac 1 {1-\eta_i} (u_i^s-u)\otimes \D\eta_i\right)\d x\\
\leq& \sum_i \int_{\Omega_i} F(\D u_i^s)(x) \eta_i\d x + \sum_i \int_{\supp\;\eta_i} \frac{\lvert \D\eta}{(1-\eta_i)^{q-1}} \lvert u_i^s-u\lvert^q\d x\\
\lesssim& \int_\Omega F(\D u)\d x + \int_{\Omega_i} \lvert u_i^s-u\rvert^q\d x.
\end{align*}
Now if $q\leq \frac{np}{n-p}$, the second term is estimated using Sobolev embedding by
\begin{align*}
\int_{\Omega_i} \lvert u_i^s-u\rvert^q\d x\lesssim \sum_i\|u_i^s-u\|_{\WW^{1,p}(\Omega_i)}^q \to 0
\end{align*}
as $s\to 1$. In particular, by a version of dominated convergence and passing to a suitable subsequence, letting first $s\to 1$, then $t\to 1$, $\F(u^{t,s})\to \F(u)$. 

If $u\in \LL^\infty$ and $q>\frac{np}{n-p}$, we employ interpolation to estimate with $\frac 1 q = \frac{(1-\theta)np}{n-p}$ (note $\theta\in(0,1)$),
\begin{align*}
\int_{\Omega_i} \lvert u_i^s-u\rvert^q\d x\lesssim& \sum_i\|u_i^s-u\|_{\WW^{1,p}(\Omega_i)}^{1-\theta} \|u_i^s-u\|_{\LL^\infty(s^{-1}\Omega_i)}^\theta + \|u_i^s-u\|_{\WW^{1,p}(\Omega_i)}\|g\|_{\WW^{1,q}(\Omega)}\\
\lesssim& \sum_i\|u_i^s-u\|_{\WW^{1,p}(\Omega_i)}^{1-\theta} \|u\|_{\LL^\infty(\Omega_i)}^\theta \to 0\\
\end{align*}
Thus, we conclude also in this case, that $\F(u^{t,s})\to \F(u)$ after passing to a suitable subsequence, letting first $s\to 1$, then $t\to 1$.

Thus, given $u\in \WW^{1,p}(\Omega)$ and $(u_j)\subset g+\WW^{1,p}_0(\Omega)\cap \WW^{1,q}_{\tp{loc}}(\Omega)$ with $u_j\to u$ in $\WW^{1,p}(\Omega)$, we can construct a diagonal subsequence $(v_j)$ by applying the above construction to $u_j$ that satisfies $v_j\to u$ in $\WW^{1,p}(\Omega)$, $\F(v_j)\to \F(u)$ and $v_j\in \WW^{1,q}(\Omega)$. In other words, Lavrentiev does not occur.

It remains to comment on how to adapt the argument to the non-autononomous setting with $q<\frac{np}{n-p}$. The only part of the argument that changes concerns $u_i^s$. We now set for $t(s),\e(s)$ to be determined,
\begin{align*}
u_i^s(x) = t(u-g)\star \rho_\e(s x)+ g.
\end{align*}
There is $c_i>0$ such that if we ensure $\e\leq c_i s$, then $u_i^s\in \WW^{1,p}_0(\Omega)$ with $u_i^s =g$ on $\p\Omega_i\cap \p\Omega$. Moreover, $u_i^s$ is $\WW^{1,q}$ in an open neighbourhood of $\p\Omega \cap \p\Omega_i$. We find, using convexity as before,
\begin{align*}
&\int_{\Omega_i} F(x,\D u_i^s)\d x\\
\leq& t \int_{\Omega_i} F(x,\D u\star \rho_\e(sx))\d x + (1-t)\int_{\Omega_i} F\left(x,\frac t {1-t}(\D g(sx)\star \rho_\e-\D g(x)\right)\d x.
\end{align*}
The second term can be estimated exactly as before. For the first term, we note that $\lvert \D u_\e\rvert\leq C\|u\|_{\WW^{1,p}(\Omega)} \e^{-\frac n p}$ by Lemma \ref{lem:mollification}. Thus, using \eqref{def:xCondition} and Lemma \ref{lem:changeOfXAlt}, since $\e \leq c_i s$,
\begin{align*}
F(x,\D u\star \rho_\e(sx))\d x \lesssim 1+ F(s x,\D u\star \rho_\e(sx))\d x.
\end{align*}
The estimate now proceeds exactly as before.

\section{Construction and convergence of approximations}\label{sec:construction}
Throughout this section $\Omega\subset \R^n$ will be an open, bounded set and $1<p\leq q$.
Let $g\in \WW^{1,q}(\Omega)$. Given $u\in g+\WW^{1,p}_0(\Omega)$, the aim of this section is to construct $u_\e\in g+\WW^{1,p}_0(\Omega)\cap \WW^{1,p}_{\tp{loc}}(\Omega)$ such that $u_\e\to u$ in $\WW^{1,p}(\Omega)$. The sequence $\{u_\e\}$ will be at the heart of our proofs of non-existence of the Lavrentiev phenomenon.

Let $\{B_i\}$ be a WB-covering of $\Omega$, where $B_i$ has radius $r_i$, and $\{\psi_i\}$ a partition of unity adapted to $\{B_i\}$ satisfying the properties of Theorem \ref{thm:partition}, in particular also \eqref{eq:partitionOfUnityDeriv}. Without loss of generality, we will always assume that all balls in the WB-covering have radius at most $1$. Fix $u\in g+\WW^{1,p}_0(\Omega)$. Introduce for $\e\in(0,1)$,
$$
u_\e = \sum_i u\star \rho_{\delta_i \e}\,\psi_i
$$
where
$\delta_i = C_0 r_i^N$ for constants $C_0,N>0$ to be made precise at a later stage, but chosen so that $(1+\delta/2)r_i+ \delta_i\leq (1+\delta)r_i$. This can be ensured if $N\geq 1$ and $C_0 = c_0 \delta$ for a sufficiently small choice of $c_0>0$.

For later use, we record that
\begin{align}\label{eq:derivative}
\D u_\e =& \sum_i \D v_{\e}\star \rho_{\delta_i \e}\,\psi_i + \sum_i u_\e\star \rho_{\delta_i \e}\,\D \psi_i \nonumber\\
=& \sum_i \left(\D v_\e\star \rho_{\delta_i \e}\right)\,\psi_i + \sum_i (v_\e\star \rho_{\delta_i \e}-v_\e)\,\D\psi_i= A_1 + A_2,
\end{align}
since $\sum \D\psi_i = 0$.

We begin by showing that $u_\e$ does indeed lie in the correct function space.
\begin{lemma}\label{lem:ueRegularity}
For any $1\leq p\leq q$, $u_\e\in g+\WW^{1,p}_0(\Omega)\cap \WW^{1,p}_{\tp{loc}}(\Omega)$.
\end{lemma}
\begin{proof}
Consider for $s>0$ the function
$$
w_s = \sum_i \tilde w_s \star \rho_{\delta_i \e}\,\psi_i
$$
where
$$
\tilde w_s(x) = \begin{cases}
			0 \qquad&\text{ if } d(x,\p\Omega)<s\\
			u &\text{ else }
			\end{cases}
$$
Note that locally $w_s$ is a finite sum of smooth functions, compactly supported in $\Omega$. Moreover, due to \eqref{eq:compareBoundary}, $w_s=0$ in an open neighbourhood of $\p\Omega$. In particular, this shows $w_s\in C_c^\infty(\Omega)$. Thus, to prove the claim it suffices to show that $w_s \to u_\e-g$ in $\WW^{1,p}_0(\Omega)$ as $s\to 0$. Note that for $s$ sufficiently small, we can guarantee that $B_i\setminus \Omega_\e\neq\emptyset$ implies $B_i\cap \Omega_s = \emptyset$, due to \eqref{eq:compareBoundary}. Moreover, note that due to \eqref{eq:compareBoundary}, if $B_i\cap\Omega_s\neq\emptyset$, then $r_i\lesssim s$. Thus, for all sufficiently small choices of $s$, using Lemma \ref{lem:mollification} and \eqref{def:WBcoverMultiplicity}, there is $c>0$ such that
\begin{align*}
\int_\Omega \lvert u_\e-g-w_s\rvert^q\d x \leq& \sum_i\int_{(1+\delta/2)B_i\cap \Omega_s} \lvert g\star \rho_{\delta_i \e}-g\rvert^q\d x\\
\lesssim& \sum_i \int_{(1+\delta)B_i\cap \Omega_{cs}}\lvert g\rvert^q\d x\\
\lesssim& \|g\|_{\WW^{1,q}(\Omega_{cs})}\xrightarrow{s\to 0} 0.
\end{align*}

We further find, using \eqref{eq:derivative}, Lemma \ref{lem:mollification} and \eqref{def:WBcoverMultiplicity}, for all sufficiently small $s$ and some $c>0$,
\begin{align*}
\int_\Omega \lvert \D u_\e-\D w_s\rvert^p\d x \leq& \sum_i \int_{(1+\delta/2)B_i\cap\Omega_s} \lvert \D u\star \rho_{\delta_i \e}-\D u\rvert^p+\lvert u\star \rho_{\delta_i\e}-g\rvert^p \lvert \D\psi_i\rvert^p\d x\\
\lesssim& \sum_i \int_{(1+\delta)B_i\cap \Omega_{cs}} \lvert \D u\rvert^p+(\delta_i\e)^p r_i^p \left(\lvert g\rvert^p+\lvert \D g\rvert^p\right)\d x\\
\lesssim& \|u\|_{\WW^{1,p}(\Omega_{cs})}\xrightarrow{s\to 0} 0.
\end{align*}
\end{proof}

Next, we turn to the convergence properties of $\{u_\e\}$ as $\e\to 0$. 
\begin{lemma}\label{lem:ueConvergence}
Suppose $1<p\leq q$. Then $u_\e\to u$ in $\WW^{1,p}(\Omega)$ as $\e\to 0$. 
\end{lemma}
\begin{proof}
We note
\begin{align*}
\|u-u_\e\|_{\LL^p(\Omega)} \leq& \sum_{i} \|u-u\star \rho_{\delta_i \e}\|_{\LL^p((1+\delta/2)B_i\setminus \Omega_\e)}+\|u-g\star \rho_{\delta_i\e}\|_{\LL^p((1+\delta/2)B_i\cap\Omega_\e}.
\end{align*}
Due to \eqref{def:WBcoverMultiplicity} and Lemma \ref{lem:mollification},
\begin{align*}
\sum_{i} \|u-u\star \rho_{\delta_i \e}\|_{\LL^p((1+\delta/2)B_i\setminus\Omega_\e)}\lesssim& \sum_{i} \delta_i \e \|u\|_{\WW^{1,p}((1+\delta)B_i)}\\
\lesssim& \e\|u\|_{\WW^{1,p}(\Omega)}\xrightarrow{\e\to 0} 0.
\end{align*}
Further, using \eqref{def:WBcoverMultiplicity} and \eqref{eq:compareBoundary}, and noting that due to \eqref{eq:compareBoundary}, if $B_i\cap\Omega_\e\neq\emptyset$, $r_i\lesssim \e$,
\begin{align*}
\sum_{i} \|u-g\star \rho_{\delta_i\e}\|_{\LL^p((1+\delta/2)B_i\cap\Omega_\e)}
\lesssim \|u\|_{\LL^p(\Omega_\e)}+\|g\|_{\LL^p(\Omega_{c\e})} \xrightarrow{\e\to 0} 0,
\end{align*}
for some $c>0$. The convergence holds since $u,g\in \WW^{1,p}(\Omega)$. In particular, we note that ${u_\e\to u\in \LL^p(\Omega)}$ and hence almost everywhere.

We now turn to establishing gradient estimates. We begin by estimating $A_1$ (recall the definition of $A_1$ in \eqref{eq:derivative}). Using \eqref{def:WBcoverMultiplicity} and Lemma \ref{lem:mollification}, we obtain
\begin{align*}
\|A_1\|_{\LL^p(\Omega)}\leq& \sum_{i} \|\D u_\e\star \rho_{\delta_i \e}\|_{\LL^p((1+\delta/2)B_i)}\lesssim \sum_{i}\|\D u_\e\|_{\LL^p((1+\delta)B_i)}\\
\lesssim& \|\D g\|_{\LL^{p}(\Omega)}+\|\D  u\|_{\LL^{p}(\Omega)}.
\end{align*}
Further, using Lemma \ref{lem:mollification} and \eqref{def:WBcoverMultiplicity},
\begin{align*}
\|A_2\|_{\LL^p(\Omega)}\leq& \sum_{i}\|\left(u_\e\star \rho_{\delta_i \e}-u_\e\right)\D\psi_i\|_{\LL^p((1+\delta/2)B_i)}\\
\lesssim& \sum_{i}r_i^{-1} \|\left(u_\e\star \rho_{\delta_i\e}-u_\e\right)\|_{\LL^p((1+\delta/2)B_i)}\\
\lesssim& \sum_{i} r_i^{-1}\delta_i \e \|u_\e\|_{\WW^{1,p}((1+\delta)B_i)}\\
\lesssim& \e\left(\|u\|_{\WW^{1,p}(\Omega)}+\|g\|_{\WW^{1,p}(\Omega)}\right).
\end{align*}
To obtain the last line, we also used that $N\geq \frac 1 n$.
Thus, using \eqref{eq:derivative} and a variant of the dominated convergence theorem, we deduce that $u_\e\to u$ in $\WW^{1,p}(\Omega)$.
\end{proof}

\section{Convergence of energies}\label{sec:energies}
In this section, we will show that in each of the scenarios of Theorem \ref{thm:noLavrentiev}, we have $\F(u_\e)\to \F(u)$. We establish this by showing ${\int_\Omega F(x,A_1)\d x \to \int_\Omega F(x,\D u)\d x}$, as well as $\int_\Omega F(x,\D u)-F(x,A_1)\d x \to 0$ as $\e\to 0$.

Throughout this section, we use the notation of Theorem \ref{thm:noLavrentiev} and assume that $\Omega$ is an open, bounded set, $1< p\leq q$ and $g\in \WW^{1,q}(\Omega)$.

Lemma \ref{lem:changeOfXAlt} is the key tool in proving convergence of $\int_\Omega F(x,A_1)\d x$.
\begin{lemma}\label{lem:A1}
Let $1\leq p\leq q< \frac{(n+\alpha)p} n$. Suppose $F$ satisfies \eqref{def:ellipticity}, \eqref{def:naturalGrowth}, \eqref{def:xHolder} and \eqref{def:xCondition}. 
Then
\begin{align*}
\int_\Omega F(x,A_1) \to \int_\Omega F(x,\D u).
\end{align*}
\end{lemma}
\begin{proof}
Note that, due to convexity of $F$ and Lemma \ref{lem:changeOfXAlt},
\begin{align*}
\int_\Omega F(x,A_1)\leq& \sum_{i} \int_{(1+\delta/2)B_i} F(x,\D u_\e\star \rho_{\delta_i \e})\psi_i\d x\\
\lesssim& 1+\sum_{i} \int_{(1+\delta/2)B_i} F(x,\D u_\e)\star \rho_{\delta_i \e}\d x\\
=& 1+\sum_{i} \int_{(1+\delta)B_i\setminus \Omega_\e} F(x,\D u)\star \rho_{\delta_i\e} \d x\\
&\quad+ \sum_{i} \int_{(1+\delta)B_i\cap\Omega_\e} F(x,\D g)\star \rho_{\delta_i\e}\d x\\
\lesssim& 1+\int_\Omega F(x,\D u)\d x.
\end{align*}
To obtain the last line, we used \eqref{def:WBcoverMultiplicity} and noted that using \eqref{eq:compareBoundary} and \eqref{def:WBcoverMultiplicity} for some $c>0$,
$$
\sum_{i}\int_{(1+\delta)B_i\cap\Omega_\e} F(\D g)\star \rho_{\delta_i\e}\d x\lesssim \int_{\Omega_{c\e}} F(\D g)\d x\xrightarrow{\e\to 0} 0,
$$ 
since $F(\D g)\in \LL^1(\Omega)$ due to \eqref{def:naturalGrowth}. By a variant of the dominated convergence theorem, the conclusion follows.
\end{proof}

\begin{remark}
Note that Lemma \ref{lem:A1} applies without restriction on $p\leq q$ whenever $F$ is autonomous, that is $F\equiv F(z)$.
\end{remark}

We now turn towards proving convergence of $\{u_\e\}$ in the non-autonomous case.
\begin{lemma}\label{lem:xDependent}
Let $1< p\leq q< \frac{(n+\alpha)p} n$. Suppose $F$ satisfies \eqref{def:ellipticity}, \eqref{def:naturalGrowth}, \eqref{def:xHolder} and \eqref{def:xCondition}. Then
\begin{align*}
\int_\Omega F(x,\D u_\e) \to \int_\Omega F(x,\D u).
\end{align*}
\end{lemma}
\begin{proof}
As a consequence of \eqref{eq:diffF}, we have
\begin{align*}
\int_\Omega F(x,\D u_\e)\d x\leq \int_\Omega F(x,A_1)+\Lambda \lvert A_2\rvert(1+\lvert A_1\rvert+\lvert A_2\rvert)^{q-1}\d x.
\end{align*}
Due to Lemma \ref{lem:A1}, we only need to consider the second term. We first consider, using Lemma \ref{lem:mollification} and \eqref{def:WBcoverMultiplicity},
\begin{align*}
\int_\Omega \lvert A_2\rvert^q\d x \leq& \sum_i \int_{(1+\delta/2)B_i}\lvert( v_\e\star \rho_{\delta_i \e}-v_\e)\D\psi_i\rvert^q\d x\\
\lesssim& \sum_i r_i^{-q}(\e\delta_i)^{1-n\left(\frac 1 p-\frac 1 q\right)}\|v_\e\|_{\WW^{1,p}((1+\delta)B_i)}^p\\
\lesssim& \e^{q-nq\left(\frac 1 p-\frac 1 q\right)}\left(\|u\|_{\WW^{1,p}(\Omega)}+\|g\|_{\WW^{1,p}(\Omega)}\right)^p\xrightarrow{\e\to 0} 0,
\end{align*}
as long as we have
\begin{align*}
1-n\left(\frac 1 p-\frac 1 q\right)> 0 \quad\Leftrightarrow\quad q<\frac{np}{n-p}
\end{align*}
and ensure $N\left(1-n\left(\frac 1 p-\frac 1 q\right)\right)\geq 1$. Further, by H\"older's inequality, we deduce that also $\int_\Omega \lvert A_2\rvert\d x\to 0$. 
Thus, it remains to consider $\int_\Omega \lvert A_2\rvert \lvert A_1\rvert^{q-1}\d x$. Here we use Lemma \ref{lem:mollification}, \eqref{eq:overlappingCubes} and \eqref{def:WBcoverMultiplicity} to find
\begin{align*}
&\int_\Omega \lvert A_1\rvert \lvert A_2\rvert^{q-1}\d x\\
\lesssim& \sum_i \sum_{\{j\colon (1+\delta)B_i\cap (1+\delta)B_j\neq\emptyset\}}\int_{(1+\delta/2)B_i} \lvert (v_\e\star \rho_{\delta_i \e}-v_\e)\D\psi_i\rvert\lvert Dv_\e \star \rho_{\delta_j \e}\psi_j\rvert^{q-1}\d x\\
\lesssim& \sum_i \sum_{\{j\colon (1+\delta/2)B_i\cap (1+\delta/2)B_j\neq\emptyset\}}r_i^{-1}\|v_\e\star \rho_{\delta_i\e}-v_\e\|_{\LL^q((1+\delta/2)B_i)}\|\D v_\e\star \rho_{\delta_j\e}\|_{\LL^q((1+\delta/2)B_j)}^{q-1}\\
\lesssim& \sum_i \sum_{\{j\colon (1+\delta/2)B_i\cap (1+\delta/2)B_j\neq\emptyset\}}r_i^{-1} (\e \delta_i)^{1-n\left(\frac 1 p-\frac 1 q\right)}\|v_\e\|_{\WW^{1,p}((1+\delta)B_i)}^\frac p q\\
&\quad\times(\e\delta_i)^{-\frac{n(q-1)(q-p)} q}\|v_\e\|_{\WW^{1,p}((1+\delta)B_j)}^\frac{(q-1)p} q\\
\lesssim& \e^{n+1-\frac{nq} p}\left(\|u\|_{\WW^{1,p}(\Omega)}+\|g\|_{\WW^{1,p}(\Omega)}\right)^p\xrightarrow{\e\to 0} 0,
\end{align*}
as long as we have
\begin{align*}
n+1-\frac{nq} p>0\quad\Leftrightarrow\quad q<\frac{(n+1)p} n
\end{align*}
and ensure $N(n+1-\frac{nq} p)>1$.
\end{proof}

In the case of autonomous functionals, we can improve on the previous result, by choosing a partition of unity that is adapted to $u$. Let $\{B_i\}$ be a WB-covering of $\Omega$, where $B_i$ has radius $r_i$, $\{\phi_i\}$ a family of functions satisfying the assumptions of Theorem \ref{thm:partition} and $\{\psi_i\}$ the corresponding partition of unity given by Theorem \ref{thm:partition}. Note that in particular \eqref{eq:partitionOfUnityDeriv} is not necessarily satisfied.
 We denote by $\{\tilde u_\e\}$ and $\{\tilde v_\e\}$ the families of functions that are the outcome of the construction of Section \ref{sec:construction} with respect to this partition of unity. The analogous decomposition to \eqref{eq:derivative}, we denote $\D \tilde u_\e = \tilde A_1 + \tilde A_2$. 

\begin{lemma}\label{lem:autonomous}
Let $1< p\leq q< \min\left(p+1,\frac{np}{n-1}\right)$. Assume $F\equiv F(z)$ satisfies \eqref{def:ellipticity} and \eqref{def:naturalGrowth}. Then there is a family $\{\phi_i\}$ satisfying the assumptions of Theorem \ref{thm:partition}, such that
\begin{align}\label{eq:autonomousConvergence}
\int_\Omega F(\D \tilde u_\e)\to \int_\Omega F(\D u).
\end{align}
Moreover, $\tilde u_\e\in g+\WW^{1,q}_0(\Omega)$ and $\tilde u_\e\to u$ in $\WW^{1,p}(\Omega)$.
\end{lemma}
\begin{proof}
Note that the proof of Lemma \ref{lem:A1} did not use \eqref{eq:partitionOfUnityDeriv}, so that Lemma \ref{lem:A1} applies to $\tilde u_\e$. Thus as in Lemma \ref{lem:xDependent}, to see \eqref{eq:autonomousConvergence},  it suffices to show 
$$
\int_\Omega \lvert \tilde A_2\rvert\left(1+\lvert \tilde A_2\rvert^{q-1}+\lvert \tilde A_2\rvert^{q-1}\right)\d x\xrightarrow{\e\to 0} 0.
$$
Let $t\in(0,1)$ to be fixed at a later stage and write $s=\frac p {p+1-q}\in(1,\infty)$, since $p\leq q<p+1$.

 Then we estimate, using H\"older's inequality,
 \begin{align*}
& \int_\Omega \lvert \tilde A_2\rvert \left(1+\lvert \tilde A_1\rvert^{q-1}+\lvert \tilde A_2\rvert^{q-1}\right)\\
\leq& \left(1+\|\tilde A_1\|_{\LL^p(\Omega)}^{q-1}+\|\tilde A_2\|_{\LL^p(\Omega)}^{q-1}\right)\left(\int_\Omega \sum_i \lvert \tilde v_\e\star \rho_{\delta_i \e}-\tilde v_\e\rvert^s\lvert \D\psi_i\rvert^s\d x\right)^\frac 1  s = I \times II.
\end{align*}
Using Lemma \ref{lem:schaffner}, \eqref{def:WBcoverMultiplicity} and \eqref{eq:overlappingCubes}, we may choose $\{\phi_j\}$ so that
\begin{align*}
II\lesssim& \sum_j \left(\int_{(1+\delta/2)B_j}\sum_{i\colon (1+\delta/2)B_i\cap (1+\delta/2)B_j\neq \emptyset}\lvert \tilde v_\e\star \rho_{\delta_i \e}-\tilde v_\e\rvert^s\lvert D\phi_j\rvert^s\d x\right)^\frac 1 s \\
 \lesssim& \sum_j r_j^{-(s+1/t)} \left(\int_{r_j}^{(1+\delta/2)r_j}\sum_{\{i\colon (1+\delta/2)B_i\cap (1+\delta/2)B_j\neq\emptyset\}} \|\tilde  v_\e\star \rho_{\delta_i\e}-\tilde v_\e\|_{\LL^s(\p B_{j,r})}^{t s}\d r\right)^\frac 1 {t s}\\
 =& II_1.
\end{align*}
Here $B_{i,r}$ denotes the ball of radius $r$ with the same center as $B_i$.

Due to \eqref{eq:overlappingCubes}, we can ensure that whenever $(1+\delta/2)B_i\cap (1+\delta_2)B_j\neq\emptyset$, then ${r_j-\delta_i>r_j/2}$. Note that for $r\in (r_j,(1+\delta/2)r_j)$, (if $p\geq n-1$, $\frac{n-1}{n-1-p}$ needs to be replaced by $\tau$ for a sufficiently large choice of $\tau$)
\begin{align*}
\|\tilde u_\e\star\rho_{\delta_i\e}-\tilde u_\e\|_{\LL^{\frac{(n-1)p}{n-1-p}}(\p B_{j,r})}\lesssim \|\tilde u_\e\star\rho_{\delta_i,\e}- \tilde u_\e\|_{\WW^{1,p}(\p B_{j,r})}
\end{align*}
and
\begin{align*}
\|\tilde u_\e\star\rho_{\delta_i\e}-\tilde u_\e\|_{\LL^p(\p B_{j,r})}\lesssim \delta_i \e \|\tilde u_\e\|_{\WW^{1,p}(A_{j,\delta\e_i})},
\end{align*}
where $A_{j,\delta_i\e}$ denotes the annulus of width $2\delta_i\e$ centered around the sphere $\p B_{j,r}$.
Thus, by interpolation
\begin{align*}
\|\tilde u_\e\star\rho_{\delta_i,\e}-\tilde u_\e\|_{\LL^\frac p {p+1-q}}\lesssim (\delta_i\e)^\theta  
\left(\|\tilde u_\e\|_{\WW^{1,p}(A_{j,\delta\e_i})}+ \|\tilde u_\e\star\rho_{\delta_i\e}-\tilde u_\e\|_{\WW^{1,p}(\p B_{j,r})}\right)
\end{align*}
where 
$$
\frac{p+1-q} p = \frac \theta p+\frac{(1-\theta)(n-1-p)}{(n-1)p} \quad\Leftrightarrow\quad \theta = 1+(n-1)\frac{p-q} p.
$$
This is valid since
$$
\theta>0 \quad\Leftrightarrow \quad q<\frac{np}{n-1}.
$$
In particular, we deduce using \eqref{eq:overlappingCubes} and with the choice $t = p+1-q\in(0,1)$ (so that $ts = p$),
\begin{align}\label{eq:II}
II_1\leq& \sum_{j}\sum_{\{i\colon (1+\delta/2)B_i\cap (1+\delta/2)B_j\neq\emptyset\}} r_j^{-\frac{p+1}{p+1-q}}(\delta_i\e)^\theta \nonumber\\
&\quad\times\int_{r_j}^{(1+\delta/2)r_j}\|\tilde v_\e\|_{\WW^{1,p}(A_{j,\delta_i\e})}^p+\|\tilde v_\e\star \rho_{\delta_i,\e}-v_\e\|_{\WW^{1,p}(\p B_r)}^p\d r\nonumber\\
\lesssim& \e^\theta\sum_j r_j^{-\frac{p+1}{p+1-q}+M\theta}\lvert B_j\rvert^{-\frac{p+1}{p+1-q}+M\theta}\|\tilde v_\e\|_{\WW^{1,p}((1+\delta)B_j)}^p\nonumber\\
\lesssim& \e^\theta\|\tilde v_\e\|_{\WW^{1,p}(\Omega)}^p,
\end{align}
if we choose $N$ large enough that $-\frac{p+1}{p+1-q}+N\theta>0$.

We now turn to estimating $I$ and note as a consequence of our argument above, since $\frac p {p+1-q}\geq p$, 
$$
\|\tilde A_2\|_{\LL^p(\Omega)}\lesssim \|\tilde A_2\|_{\LL^s(\Omega)}^\frac p s\xrightarrow{\e\to 0}0.
$$
Further, noting that our estimate on $A_1$ in the proof of Lemma \ref{lem:ueConvergence} does not use \eqref{eq:partitionOfUnityDeriv}, we obtain that 
$$
\|\tilde A_1\|_{\LL^p(\Omega)}\lesssim \|\D u\|_{\LL^{1,p}(\Omega)}+\|\D g\|_{\LL^p(\Omega)}.
$$
Thus, combining \eqref{eq:II} with these observations, \eqref{eq:autonomousConvergence} follows. 

Further, using the argument of Lemma \ref{lem:ueConvergence}, we deduce that $\tilde u_\e\to u$ in $\WW^{1,p}(\Omega)$. Thus, it remains to show that $\tilde u_\e \in g+\WW^{1,q}_0(\Omega)$. Arguing as in Lemma \ref{lem:ueRegularity}, it suffices to estimate
\begin{align*}
\sum_i \int_{(1+\delta/2)B_i\cap \Omega_s} \lvert \tilde v_\e\star \rho_{\delta_i \e}-\tilde v_\e\rvert^q \lvert \D\psi_i\rvert^q\d x\xrightarrow{s\to 0} 0.
\end{align*}
However, noting that $q\leq\frac p {p+1-q}$, this again follows from our estimate on $II$ using H\"older's inequality.
\end{proof}

We next consider controlled growth-conditions where we need to restrict to $2\leq p$. 
We prove the following:
\begin{lemma}\label{lem:controlled}
Let $2\leq p\leq q< \min\left(p+2,p\left(1+\frac 2 {n-1}\right)\right)$. Assume $F\equiv F(z)$ is satisfies \eqref{def:ellipticity} and \eqref{def:controlledGrowth}. Suppose $u\in g+\WW^{1,p}_0(\Omega)$ solves \eqref{eq:min}. Then there is a family $\{\phi_i\}$ satisfying the assumptions of Theorem \ref{thm:partition}, such that
\begin{align*}
\int_\Omega F(\D \tilde u_\e)\to \int_\Omega F(\D u).
\end{align*}
Moreover, $\tilde u_\e\in g+\WW^{1,q}_0(\Omega)$ and $\tilde u_\e \to u$ in $\WW^{1,p}(\Omega)$.
\end{lemma}
\begin{proof}
Using the same notation as in Lemma \ref{lem:autonomous}, we find
\begin{align*}
&\int_\Omega F(\D \tilde u_\e)-\int_\Omega F(\tilde A_1)-\int_\Omega \p_z F(\tilde A_1)\cdot (\D \tilde u_\e-\tilde A_1)\\
\lesssim& \int_\Omega \lvert \tilde A_2\rvert^2(1+\lvert \tilde A_1\rvert+\lvert \tilde A_2\rvert)^{q-2}.
\end{align*}

We first deal with the term on the right-hand side, which we'll prove converges to $0$ as $\e\to 0$. Using H\"older, we find
\begin{align*}
&\int_\Omega \lvert \tilde A_1\rvert^2\left(1+\rvert \tilde A_1\rvert^{q-2}+\lvert \tilde A_2\rvert^{q-2}\right)\d x\\
\leq& \left(1+\|\tilde A_1\|_{\LL^p(\Omega)}+\| \tilde A_2\|_{\LL^p(\Omega)}\right)^{q-2}\left(\int_\Omega \left(\sum_i \lvert \tilde v_\e\star \rho_{\delta_i \e}-\tilde v_\e\lvert\right)^\frac{2p} {p+2-q}\right)^\frac{p+2-q} p.
\end{align*}
We can now proceed exactly as in Lemma \ref{lem:autonomous} as long as $N$ is sufficiently large and $\theta>0$, where
$$
\frac{p+2-q} {2p}=\frac \theta p+\frac{(1-\theta)(n-1-p)}{(n-1)p} \quad\Leftrightarrow \quad\theta = \frac{n+1} 2-\frac{(n-1)q}{2p},
$$
noting that
\begin{align*}
\theta>0 \quad\Leftrightarrow\quad q<p\left(1+\frac{2}{n-1}\right).
\end{align*}

Moreover, arguing as in Lemma \ref{lem:autonomous} it also follows that $\tilde u_\e\in g+\WW^{1,q}_0(\Omega)$ and $\tilde u_\e\to u$ in $\WW^{1,p}(\Omega)$.

Thus, it remains to establish that
\begin{align*}
\int_\Omega \p_z F(\tilde A_1)\cdot(\D \tilde u_\e-\tilde A_1)\to 0.
\end{align*}
Note that since $q\leq \frac{p}{p+2-q}$, our argument above showed that $\D \tilde u_\e-\tilde A_1 = \tilde A_2\to 0$ in $\LL^q(\Omega)$. In particular, using Lemma \ref{lem:EulerLagrange}, it suffices to show $\p_z F(\tilde A_1)\to \p_z F(\D u)$ in $\LL^{q^\prime}(\Omega)$, since then by H\"older's inequality
\begin{align*}
\int_\Omega \p_z F(\tilde A_1)\cdot(\D \tilde u_\e-\tilde A_1)\leq& \|\p_z F(\tilde A_1)\|_{\LL^{q^\prime}(\Omega)}\|\tilde A_2\|_{\LL^q(\Omega)}\\
\leq&\left(1+\|\p_z F(\D u)\|_{\LL^{q^\prime}(\Omega)}\right) \|\tilde A_2\|_{\LL^q(\Omega)}\to 0.
\end{align*}

Thus, we turn to proving $\p_z F(\tilde A_1)\to \p_z F(\D u)\in \LL^{q^\prime}(\Omega)$. Using Fenchel's inequality, we find for any $s>1$,
\begin{align*}
F^\ast(\p_z F(\tilde A_1)) = \langle \p_z F(\tilde A_1),\tilde A_1\rangle - F(\tilde A_1)\leq \frac 1 s F^\ast(\p_z F(\tilde A_1)+\frac 1 s F(s \tilde A_1)-F(\tilde A_1).
\end{align*}
Re-arranging, we find, using also that as a consequence of \eqref{def:ellipticity}, $F(z)\geq -c$ for some $c>0$,
\begin{align*}
F^\ast(\p_z F(\tilde A_1)) \leq c+\frac 1 {s-1} F(s \tilde A_1). 
\end{align*}

Re-calling the definition of $\tilde A_1$, using Jensen's inequality, \eqref{def:WBcoverMultiplicity} and employing Lemma \ref{lem:mollification}, we obtain
\begin{align*}
\int_\Omega F(s \tilde A_1)\d x \leq& \sum_i\int_{(1+\delta/2)B_i} F(s \D \tilde v_\e)\star \rho_{\e\delta_i}\psi_i\d x\\
\lesssim& 1+ \int_{\Omega} F(s\D u)\d x+\int_\Omega F(s\D g)\d x\\
\lesssim& 1+\int_\Omega F(\D u)+F(\D g)\d x<\infty.
\end{align*}
To obtain the last line, we used the doubling property of $F$. Since $\tilde A_1\to \D u$ almost everywhere by Lemma \ref{lem:ueConvergence}, this concludes the proof, by an application of a version of dominated convergence.
\end{proof}

We next turn to controlled-duality growth.
\begin{lemma}\label{lem:controlledDuality}
Let $2\leq p\leq q< \frac{np}{n-p}$. Assume $F\equiv F(z)$ satisfies \eqref{def:ellipticity} and \eqref{def:controlledDualityGrowth}. Suppose $u\in g+\WW^{1,p}_0(\Omega)$ solves \eqref{eq:min}. Then
\begin{align*}
\int_\Omega F(\D u_\e)\to \int_\Omega F(\D u).
\end{align*}
\end{lemma}
\begin{proof}
Arguing as in Lemma \ref{lem:controlledDuality}, it suffices to estimate
\begin{align}\label{eq:controlled1}
\int_\Omega \lvert A_2\rvert^2 (1+\lvert\p_z F(\D u)\rvert+\lvert\p_z F(A_2)\rvert)^\frac{q-2}{q-1},
\end{align}
and 
\begin{align}\label{eq:controlled2}
\int_\Omega \p_z F(A_1)\cdot A_2\d x.
\end{align}
On the one hand,
\begin{align*}
\int_\Omega \lvert A_2\rvert^2 (1+\lvert\p_z F(A_2)\rvert)^\frac{q-2}{q-1}\lesssim \int_\Omega \lvert A_2\rvert^2(1+\lvert A_2\rvert)^{q-2}, 
\end{align*}
 since a consequence of the convexity of $F$ and \eqref{def:ellipticity} is that $\lvert \p_z F(z)\rvert\lesssim 1+\lvert z\rvert^{q-1}$. However, as in Lemma \ref{lem:xDependent}, we note that $\int_\Omega \lvert A_2\rvert^2(1+\lvert A_2\rvert)^{q-2}\to 0$ as $\e\to 0$ since $2\leq q<\frac{np}{n-p}$.
On the other hand, due to Lemma \ref{lem:EulerLagrange}, we may use H\"older's inequality to estimate
\begin{align*}
\int_\Omega \lvert A_2\rvert^2 \lvert\p_z F(\D u)\rvert^\frac{q-2}{q-1}\leq \|A_2\|_{\LL^q(\Omega)}^2 \|\p_z F(\D u)\|_{\LL^{q^\prime}(\Omega)}^{q-2}.
\end{align*}
Noting that due to Lemma \ref{lem:EulerLagrange}, $\p_z F(\D u)\in \LL^{q^\prime}(\Omega)$, using again that $\|A_2\|_{\LL^q(\Omega)}\to 0$, this concludes the proof of \eqref{eq:controlled1}. 

On the other hand, \eqref{def:controlledDualityGrowth} implies that there exists $c>0$ such that for any $x,y\in\Omega$ and $\tau\in[0,1]$,
\begin{align*}
F^\ast(\tau x+(1-\tau y)\leq \tau F^\ast(x)+(1-\tau) F^\ast(y)-\tau(1-\tau)c (1+\lvert x\rvert+\lvert y\rvert)^{q^\prime-2}\lvert x-y\rvert^2
\end{align*}
For a proof of this fact, see e.g. \cite{DeFilippis2020} or \cite{Hiriart-Urruty1993a}. Now, using Fenchel's inequality and the fact that $\lvert z\rvert^{q^\prime}-1\lesssim F^\ast(z)\lesssim 1+\lvert z\rvert^{p^\prime}$,
\begin{align*}
F^\ast(\p_z F(A_1)) \leq& 2 F^\ast\left(\frac 1 2\p_z F(A_1)\right) + 2 F(A_1)\\
\lesssim& 1+F^\ast(\p_z F(A_1))- \frac c 4 (1+\lvert \p_z F(A_1)\rvert)^{q^\prime-2}\lvert \p_z F(A_1)\rvert^2 + 2 F(A_1).
\end{align*}
Re-arranging and integrating, we deduce
\begin{align*}
\int_\Omega (1+\lvert \p_z F(A_1)\rvert)^{q^\prime-2}\lvert \p_z F(A_1)\rvert^2\lesssim 1+\int_\Omega F(A_1)\d x\lesssim 1+\int_\Omega F(\D u)\d x.
\end{align*}
To obtain the last inequality, we used Lemma \ref{lem:A1}. We now estimate for any $\tau>0$,
\begin{align*}
&\int_\Omega \p_z F(A_1)\cdot A_2\d x\\
\leq& \tau\int_\Omega (1+\lvert \p_z F(A_1)\rvert)^{q^\prime-2}\lvert \p_z F(A_1)\rvert^2\d x+C(\tau)\int_\Omega (1+\lvert A_2\rvert)^{q-2}\lvert A_2\rvert^2\d x\\
\lesssim& \tau\left(1+\int_\Omega F(\D u)\d x\right)+c(\tau)\left(\int_\Omega \lvert A_2\rvert^q\d x+\left(\int_\Omega \lvert A_2\rvert^q\d x\right)^\frac 2 q\right).
\end{align*}
Since $A_2\to 0 $ in $\LL^q(\Omega)$, we deduce that
\begin{align*}
\int_\Omega \p_z F(A_1)\cdot A_2\d x\to 0,
\end{align*}
which concludes the proof.
\end{proof}

Finally, we turn to the case where we make the a-priori assumption that $u\in \LL^\infty(\Omega)$. 
\begin{lemma}\label{lem:Linfty}
Let $1\leq p\leq q< p+1$. Assume $F\equiv F(z)$ satisfies \eqref{def:ellipticity} and \eqref{def:naturalGrowth}. Further suppose $u\in \LL^\infty(\Omega)$. Then
\begin{align*}
\int_\Omega F(\D u_\e)\to \int_\Omega F(\D u).
\end{align*}
If $F\equiv F(z)$ satisfies \eqref{def:ellipticity}, \eqref{def:controlledGrowth}, \eqref{def:doubling} and $u$ solves \eqref{eq:min}, the conclusion holds under the restriction $2\leq p\leq q<p+2$.
\end{lemma}
\begin{proof}
As in Lemma \ref{lem:xDependent} it suffices to deal with $\int_\Omega \lvert A_2\rvert \lvert A_1\rvert^{q-1}$. Using \eqref{lem:mollification}, H\"older's inequality (note $\frac p {p+q-1}\leq q$), \eqref{def:WBcoverMultiplicity} and \eqref{eq:overlappingCubes} we find
\begin{align*}
&\int_\Omega \lvert A_2\rvert \lvert A_1\rvert^{q-1}\\
\lesssim& \|\D v_\e\|_{\LL^p(\Omega)}^{q-1}\sum_{j\in I} \sum_{i\colon (1+\delta/2)B_i\cap (1+\delta/2)B_j\neq \emptyset}\Big(\int_{\Omega_\e} \lvert g\star \rho_{\delta_i,\e}-g\rvert^\frac p {p+1-q}|\D\psi_i|^\frac p {p+1-q}\\
&\quad+\int_{\Omega\setminus\Omega_\e} \lvert u\star \rho_{\delta_i,\e}-u\rvert^\frac p {p+1-q}|\D\psi_i|^\frac p {p+1-q}\Big)^\frac{p+1-q} p\\
\lesssim& \|\D v_\e\|_{\LL^p(\Omega)}^{q-1}\sum_{j\in I} \sum_{i\colon (1+\delta/2)B_i\cap (1+\delta/2)B_j\neq \emptyset} r_i^{-1}\Big((\delta_i \e)\|g\|_{\WW^{1,q}((1+\delta)B_i\cap\Omega_\e)}\\
&\quad+(\delta_i\e)^\frac p {p+1-q} \|u\|_{\LL^\infty((1+\delta)B_i\setminus\Omega_\e)}^{1-\frac p {p+1-q}}\|u\|_{\WW^{1,p}((1+\delta)B_i\setminus\Omega_\e)}^\frac p {p+1-q}\Big)\\
\lesssim& \e^\frac p {p+1-q} \|\D v_\e\|_{\LL^p(\Omega)}^{q-1}\left(\|g\|_{\WW^{1,q}(\Omega)}+\|u\|_{\LL^\infty(\Omega)}^{1-\frac p {p+1-q}}\|u\|_{\WW^{1,p}(\Omega)}^\frac p {p+1-q}\right)\to 0
\end{align*}
provided $N$ is sufficiently large that $-1 + N\frac p {p+1-q}\geq 0$.

The proof under \eqref{def:controlledGrowth} is nearly idential. First, it suffices to consider $\int_\Omega \lvert A_2\rvert^2 \lvert A_1\rvert^{q-2}$ arguing as in Lemma \ref{lem:controlled}. This can be estimated similar to the above argument.
\end{proof}

%\begin{remark}
%We note that the previous result also applies to autonomous integrands for which the outcome of Lemma \ref{lem:A1} is known by other means. This includes in particular the important class of functionals of double-phase type. Further examples in this direction can be found in \cite{Bulicek2022}.
%\end{remark}

\begin{remark}\label{rem:xDep}
We note that we only use the fact that $u\in \LL^\infty(\Omega)$ in order to establish the interpolative bound
$$
\|u_\e-u\|_{\LL^p(B)}\lesssim \e^\frac p q \|u\|_{\WW^{1,p}(B+\e B_1)}^\frac p q\|u\|_{\LL^\infty(B+\e B_1)}^{1-\frac p q}.
$$
Thus, $\LL^\infty$ may be replaced by $\mathrm{BMO}(\Omega)$ in Lemma \ref{lem:Linfty}.
\end{remark}

The assumption that $u\in \LL^\infty(\Omega)$ is especially relevant in the scalar case $m=1$. In this case, we may assume without loss of generality that $u\in \LL^\infty(\Omega)$. This follows using an observation made in \cite{Bulicek2022}.
\begin{lemma}\label{lem:scalar}
Let $1<p$ and $m=1$. Assume $F$ satisfies \eqref{def:ellipticity} and \eqref{def:naturalGrowth}. Then there exists $\{u_k\}\subset g+\left(\WW^{1,p}_0(\Omega)\cap\LL^\infty(\Omega)\right)$ such that $u_k\to u\in \WW^{1,p}(\Omega)$ and $\F(u_k)\to \F(u)$.
\end{lemma}
\begin{proof}
Consider
$$
T_k u = \begin{cases}
	u \quad &\text{ if } \lvert u-g\rvert \leq k\\
	g+\frac{u-g}{\lvert u-g\rvert} k &\text{ else}.
	\end{cases}
$$
Then $T_k u\to u$ in $\WW^{1,p}(\Omega)$. Moreover,
\begin{align*}
\int_\Omega F(x,\D T_k u) = \int_\Omega F(x,\D g+\D (u-g) \,1_{\lvert u-g\rvert \leq k}).
\end{align*}

\begin{align*}
\int_\Omega F(x,\D g+\D (u-g) \,1_{\{\lvert u-g\rvert \leq k\}}) =& \int_{\Omega\cap \{\lvert u-g\rvert \leq k\}} F(x,\D u)+\int_{\Omega\cap \{\lvert u-g\rvert\geq k\}}F(x,\D g) \\
\lesssim& \int_\Omega F(x,\D u)+\int_\Omega 1+\lvert \D g\rvert^q \d x<\infty.
\end{align*}
Note that due to \eqref{def:ellipticity}, $F(x,\cdot)\geq -C_0$ for some $C_0\geq 0$ and almost every $x\in\Omega$. In particular, we conclude by a variant of the dominated convergence theorem that $\int_\Omega F(x,\D T_k u)\to \int_\Omega F(x,\D u).$ This concludes the proof.
\end{proof}
Having Lemma \ref{lem:scalar} at hand, we can apply Lemma \ref{lem:Linfty}, Remark \ref{rem:xDep} and apply a diagonal subsequence argument to obtain the following statement:
\begin{lemma}
Suppose $u\colon \Omega \to \R$. Let $1<p\leq q< p+1$. Suppose $F\equiv F(z)$ satisfies \eqref{def:ellipticity} and \eqref{def:naturalGrowth} and the outcome of Lemma \ref{lem:A1} holds. Then there exists $u_k\in g+\WW^{1,q}_0(\Omega)$ such that $u_k \to u$ in $\WW^{1,p}(\Omega)$ and $\F(u_k)\to \F(u)$ as $k\to\infty$.

If $F\equiv F(z)$ satisfies \eqref{def:ellipticity} and \eqref{def:controlledGrowth} it suffices to assume $2\leq p\leq q<p+2$.
\end{lemma}

\begin{remark}
Lemma \ref{lem:scalar} and hence also the cases \ref{eq:scalarnatural} and \ref{eq:scalarcontrolled} in Theorem \ref{thm:noLavrentiev} easily generalise to integrands with Uhlenbeck structure $F\equiv F(\lvert \cdot\rvert)$, see also \cite[Section 7]{Bulicek2022}.
\end{remark}

\section{Proof of Theorem \ref{thm:noLavrentiev}}\label{sec:theoremProof}
We have now collected all ingredients we need in order to prove Theorem \ref{thm:noLavrentiev}.
\begin{proof}[Proof of Theorem \ref{thm:noLavrentiev}]
Due to the definition in order to prove $\F=\overline \F$ in case \ref{eq:nonautonomous}--\ref{eq:scalarnatural}, it suffices for a fixed $u\in g+\WW^{1,p}_0(\Omega)$ to exhibit a sequence $\{u_\e\}\subset g+\WW^{1,q}_0(\Omega)$ such that $u_\e\rightharpoonup u$ weakly in $\WW^{1,p}(\Omega)$ and $\F(u_\e)\to \F(u)$. We claim that in the cases \ref{eq:nonautonomous}, \ref{eq:Linftynatural} and \ref{eq:scalarnatural} the sequence $\{u_\e\}$ constructed in Section \ref{sec:construction} has these properties. Due to Lemma \ref{lem:ueRegularity} and Lemma \ref{lem:ueConvergence} the regularity and convergence properties regarding $\{u_\e\}$ are satisfied. The fact that $\F(u_\e)\to \F(u)$ follows from Lemma \ref{lem:xDependent} and Lemma \ref{lem:Linfty} in the case \ref{eq:nonautonomous} and \ref{eq:Linftynatural}, respectively. In the case \ref{eq:scalarnatural} it follows from combining Lemma \ref{lem:scalar} with Lemma \ref{lem:Linfty}. The remaining case \ref{eq:natural} follows from Lemma \ref{lem:autonomous} using the sequence $\tilde u_\e$ constructed there.

In order to prove that the Lavrentiev phenomenon does not occur it suffices to prove that for $u\in g+\WW^{1,p}_0(\Omega)$ solving \eqref{eq:min}, it holds that $\F(u)=\overline\F(u)$. Thus it suffices to exhibit for such $u$ a sequence $\{u_\e\}\subset g+\WW^{1,q}_0(\Omega)$ such that $u_\e\rightharpoonup u$ weakly in $\WW^{1,p}(\Omega)$ and $\F(u_\e)\to \F(u)$. We claim again that the sequence $\{u_\e\}$ constructed in Section \ref{sec:construction} has these properties in the cases \ref{eq:controlledDuality}--\ref{eq:Linftycontrolled}. As before the regularity and convergence properties regarding $\{u_\e\}$ are satisfied. The fact that $\F(u_\e)\to\F(u)$ follows from Lemma \ref{lem:controlledDuality} in the case \ref{eq:controlledDuality}. In the case \ref{eq:Linftycontrolled} it follows from Lemma \ref{lem:Linfty}, while the case \ref{eq:scalarcontrolled} follows from combining Lemma \ref{lem:scalar} and Lemma \ref{lem:Linfty}. The only remaining case is \eqref{eq:controlled}, which follows from Lemma \ref{lem:controlled}.

Thus we have proven all cases in Theorem \ref{thm:noLavrentiev}.
\end{proof}

\vfill
\pagebreak

\appendix

\bibliographystyle{siam}

%\bibliography{../Refs/pqboundary}
%\bibliography{../bibtex/Lavrentiev}
%\bibliography{../../Citations/My-Library.bib}
%\bibliography{pqboundary}

\begin{thebibliography}{10}

\bibitem{Acerbi2003}
{\sc E.~Acerbi, G.~Bouchitt{\'{e}}, and I.~Fonseca}, {\em {Relaxation of convex
  functionals: The gap problem}}, Ann. l'Institut Henri Poincare Anal. Non
  Lineare, 20 (2003), pp.~359--390.

\bibitem{Acerbi1994a}
{\sc E.~Acerbi and G.~{Dal Maso}}, {\em {New lower semicontinuity results for
  polyconvex integrals}}, Calc. Var. Partial Differ. Equ., 2 (1994),
  pp.~329--371.
  
\bibitem{Borowski2022}
{\sc M.~Borowski, I.~Chlebicka, B.~Miasojedow}, {\em {Absence of Lavrentiev's gap for anisotropic functionals}}, arXiv Prepr. arXiv2210:15217, (2022).
  

\bibitem{Diening2020}
{\sc A.~Balci, L.~Diening, and M.~Surnachev}, {\em {New Examples on Lavrentiev
  Gap Using Fractals}}, Calc. Var. Partial Differ. Equations2, 59 (2020).

\bibitem{Balci2021}
{\sc A.~Balci and M.~Surnachev}, {\em {Lavrentiev gap for some classes of
  generalized Orlicz functions}}, Nonlinear Anal., 207 (2021), p.~112329.

\bibitem{Bulicek2022}
{\sc M.~Bul{\'{i}}{\v{c}}ek, P.~Gwiazda, and J.~Skrzeczkowski}, {\em {On a
  Range of Exponents for Absence of Lavrentiev Phenomenon for Double Phase
  Functionals}}, Arch. Ration. Mech. Anal., 246 (2022), pp.~209--240.

\bibitem{Buttazo1995}
{\sc G.~Buttazo and M.~Belloni}, {\em {A survey of old and recent results about
  the gap phenomenon in the Calculus of Variations}}, Math. Appl., 331 (1995),
  pp.~1--27.

\bibitem{Buttazo1992}
{\sc G.~Buttazo and V.~Mizel}, {\em {Interpretation of the Lavrentiev
  phenomenon by relaxation}}, J. Funct. Anal., 2 (1992), pp.~434--460.

\bibitem{Carozza2011}
{\sc M.~Carozza, J.~Kristensen, and A.~{Passarelli di Napoli}}, {\em {Higher
  differentiability of minimizers of convex variational integrals}}, Ann.
  l'Institut Henri Poincare Anal. Non Lineare, 28 (2011), pp.~395--411.

\bibitem{Carozza2013}
\leavevmode\vrule height 2pt depth -1.6pt width 23pt, {\em {Regularity of
  minimisers of autonomous convex variational integrals}}, Ann. della Scu.
  Norm. Sup. di Pisa, 13 (2013).

\bibitem{DeFilippis2020}
{\sc C.~de~Filippis, L.~Koch, and J.~Kristensen}, {\em {Regularity in relaxed
  convex problems}}, {in Prep.},  (2022).

\bibitem{deFilippis2021c}
{\sc C.~de~Filippis and G.~Mingione}, {\em {Interpolative gap bounds for
  nonautonomous integrals}}, Anal. Math. Phys., 11 (2021).

\bibitem{deFilippis2022}
{\sc F.~{De Filippis} and F.~Leonetti}, {\em {No Lavrentiev gap for some double
  phase integrals}}, Adv. Calc. Var.,  (2022).

\bibitem{Dyda2022}
{\sc B.~Dyda and M.~Kijaczko}, {\em {On density of compactly supported smooth
  functions in fractional Sobolev spaces}}, Ann. di Mat. Pura Appl., 201
  (2022), pp.~1855--1867.

\bibitem{Esposito2019}
{\sc A.~Esposito, F.~Leonetti, and P.~Petricca}, {\em {Absence of Lavrentiev
  gap for non-autonomous functionals with (p,q)-growth}}, Adv. Nonlinear Anal.,
  8 (2019), pp.~73--78.

\bibitem{Esposito2004}
{\sc L.~Esposito, F.~Leonetti, and G.~Mingione}, {\em {Sharp regularity for
  functionals with (p,q) growth}}, J. Differ. Equations, 204 (2004), pp.~5--55.

\bibitem{evans}
{\sc L.~Evans}, {\em {Partial Differential Equations}}, American Mathematical
  Society, 2nd~ed., 2015.

\bibitem{Fonseca1997}
{\sc I.~Fonseca and J.~Mal{\'{y}}}, {\em {Relaxation of multiple integrals
  below the growth exponent for the energy density}}, Ann. l'Institut Henri
  Poincare Anal. Non Lineare, 14 (1997), pp.~309--338.

\bibitem{Fonseca2005a}
\leavevmode\vrule height 2pt depth -1.6pt width 23pt, {\em {From Jacobian to
  Hessian: distributional form and relaxation}}, Riv. Mat. Univ. Parma, 4
  (2005), pp.~45--74.

\bibitem{Foss2001}
{\sc M.~Foss}, {\em {On Lavrentiev's phenomenon}}, PhD thesis, Carnegie Mellon
  University, 2001.

\bibitem{Foss2003}
\leavevmode\vrule height 2pt depth -1.6pt width 23pt, {\em {The Lavrentiev gaph
  phenomenon in nonlinear elasticity}}, Arch. Ration. Mech. Anal., 167 (2003),
  pp.~336--365.

\bibitem{Giusti2003}
{\sc E.~Giusti}, {\em {Direct Methods in the Calculus of Variations}}, World
  Scientific, 2003.

\bibitem{Grisvard1992}
{\sc P.~Grisvard}, {\em {Elliptic Problems in Nonsmooth Domains}}, vol.~22,
  Society for Industrial and Applied Mathematics, University City,
  Philadelphia, 1992.

\bibitem{Harjulehto2018}
{\sc P.~Harjulehto, P.~H{\"{a}}st{\"{o}}, and A.~Karppinen}, {\em {Local higher
  integrability of the gradient of a quasiminimizer under generalized Orlicz
  growth conditions}}, Nonlinear Anal., 177 (2018), pp.~543--552.

\bibitem{Hiriart-Urruty1993a}
{\sc J.~B.~Hiriart-Urruty, and C.~Lemar}, {\em {Convex Analysis and Minimization Algorithms II}}, Springer-Verlag Berlin Heidelberg, Berlin, 1993.


\bibitem{Hirsch2020}
{\sc J.~Hirsch and M.~Sch{\"{a}}ffner}, {\em {Growth conditions and regularity,
  an optimal local boundedness result}}, Commun. Contemp. Math., 23 (2020),
  p.~2050029.

\bibitem{Hong1992}
{\sc M.~Hong}, {\em {Some remarks on the minimizers of variational integrals
  with (p,q) growth conditions}}, J. Differ. Equations, 6 (1992), pp.~91--101.

\bibitem{Kislyakov2005}
{\sc S.~Kislyakov and N.~Kruglyak}, {\em {Stability of approximation under
  singular integrals, and Cald\'{e}ron-Zygmund type decompositions}}, PDMI
  Prepr.,  (2005).

\bibitem{Kislyakov2013}
\leavevmode\vrule height 2pt depth -1.6pt width 23pt, {\em {Extremal Problems
  in Interpolation Theory, Whitney-Besicovitch Coverings and Singular
  Integrals}}, Monogr. Mat., 74 (2013), pp.~663--714.

\bibitem{Koch2020}
{\sc L.~Koch}, {\em {Global higher differentiability for minimisers of convex
  functionals with (p,q)-growth}}, Calc. Var. Partial Differ. Equ., 60 (2021).

\bibitem{Koch2021a}
\leavevmode\vrule height 2pt depth -1.6pt width 23pt, {\em {Global higher
  integrability for minimisers of convex obstacle problems with (p,q)-growth}},
  Calc. Var. Partial Differ. Equ., 60 (2021).

\bibitem{Koch2022}
{\sc L.~Koch and J.~Kristensen}, {\em {On the validity of the Euler-Lagrange
  system without growth assumptions}}, arXiv Prepr. arXiv2203.00333,  (2022).

\bibitem{Lavrentiev1926}
{\sc M.~Lavrentiev}, {\em {Sur quelques probl{\`{e}}me du calcul des
  variations}}, Ann. di Mat. Pura Appl., 4 (1926), pp.~7--28.

\bibitem{Marcellini1989}
{\sc P.~Marcellini}, {\em {Regularity of minimizers of integrals of the
  calculus of variations with non-standard growth conditions}}, Arch. Ration.
  Mech. Anal., 105 (1989), pp.~267--284.

\bibitem{Marcellini1991}
\leavevmode\vrule height 2pt depth -1.6pt width 23pt, {\em {Regularity and
  existence of solutions of elliptic equations with p,q-growth conditions}}, J.
  Differ. Equations, 90 (1991), pp.~1--30.

\bibitem{Mingione2006}
{\sc G.~Mingione}, {\em {Regularity of minima: an invitation to the dark side
  of the calculus of variations}}, Appl. Math, 51 (2006), pp.~355--426.

\bibitem{Mingione2021}
{\sc G.~Mingione and V.~D. R\u{a}dulescu}, {\em {Recent developments in
  problems with nonstandard growth and nonuniform ellipticity}}, J. Math. Anal.
  Appl.,  (2021).

\bibitem{Schaeffner2020}
{\sc M.~Sch{\"{a}}ffner}, {\em {Higher Integrability for variational integrals
  with non-standard growth}}, Calc. Var. Partial Differ. Equ., 60 (2021).

\bibitem{Zhikov1993}
{\sc V.~Zhikov}, {\em {Lavrentiev phenomenon and homogenization for some
  variational problems}}, C. R. Acad. Sci. Paris S\'{e}r. Mat., 50 (1993),
  pp.~674--710.

\bibitem{Zhikov1995}
\leavevmode\vrule height 2pt depth -1.6pt width 23pt, {\em {On Lavrentiev's
  Phenomenon}}, Russ. J. Math. Phys., 3 (1995), pp.~249--269.

\bibitem{Zhikov1987}
{\sc V.~V. Zhikov}, {\em {Averaging of functionals of the calculus of
  variations and elasticity theory}}, Izv. Math., 29 (1987), pp.~33--66.

\end{thebibliography}
\end{document}